\documentclass[10pt, reqno]{amsart}

\usepackage{preamble}
\usetikzlibrary{arrows.meta,positioning,fit,calc,shapes.multipart}

\newcommand{\lk}{\ell k}

\title[Steenrod operations and symplectic arithmetic duality]{Steenrod operations and \\ symplectic arithmetic duality}

\author{Tony Feng}

\begin{document}  

\begin{abstract}
This expository article elaborates upon my talk at the 2025 AMS Summer Institute on Algebraic Geometry. It gives an introduction to a conjecture from Tate's 1966 S\'eminaire Bourbaki report, predicting the existence of a symplectic form on Brauer groups of surfaces over finite fields, and then an informal tour of the proof in \cite{Feng20} and \cite{CF}. 
\end{abstract}

\maketitle

\begin{figure}[!h]
    \centering
    \includegraphics[scale=.45]{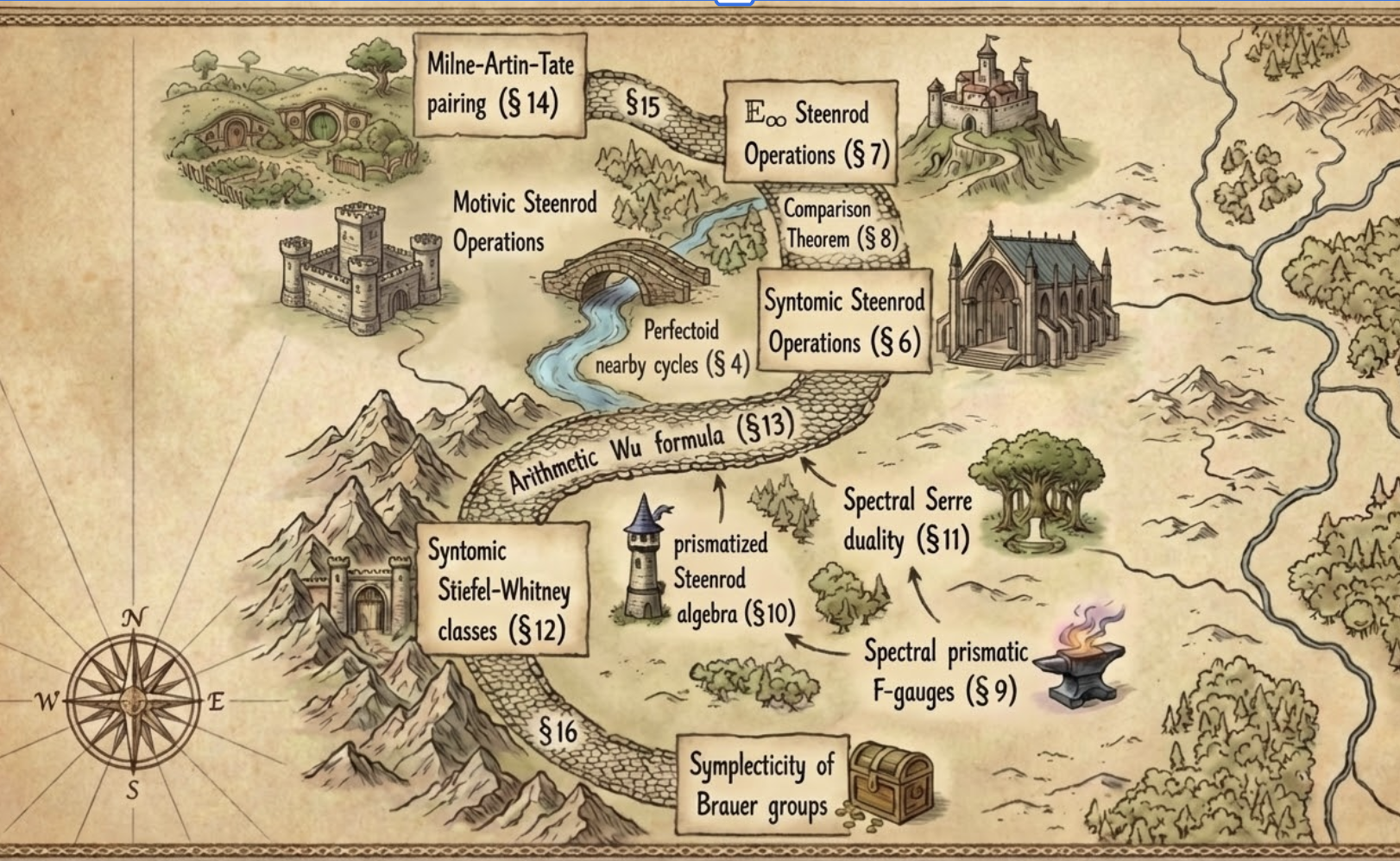}
    \caption{This roadmap (animated by Nano Banana Pro) depicts the long, winding journey to Theorem \ref{thm:symplecticityatp}.}\label{fig:roadmap}
\end{figure}

\tableofcontents

\section{Introduction}

This expository article surveys a Conjecture from Tate's 1966 Bourbaki report \cite{Tate66}, predicting the existence of a symplectic form on the Brauer group of smooth projective surfaces over finite fields. After a colorful history, with contributions from many mathematicians, the Conjecture has finally been resolved in the recent works \cite{Feng20} and \cite{CF} of the author and Shachar Carmeli. 

\subsection{Arithmetic duality} Grothendieck introduced general foundations of sheaf cohomology, and constructed in particular the \emph{\'etale cohomology} theory, which provided a cohomology theory for algebraic geometry analogous to singular cohomology of topological spaces. This allows one to assign cohomological invariants to schemes, and in particular to commutative rings (through their associated affine schemes). 

Miraculously, many rings of number-theoretic interest---such as the integers $\Z$, the rational numbers $\Q$, the $p$-adic fields $\Q_p$, and the finite fields $\F_p$---exhibit cohomological features reminiscent of those of manifolds. In particular, \emph{arithmetic duality}, announced by Tate in his 1962 ICM address \cite{Tate62}, refers to the phenomenon that the cohomology groups of global and local fields enjoy a duality theory analogous to Poincar\'e duality. 

How did Tate discover arithmetic duality? He seems to have been motivated by a result of Cassels, who had constructed a symplectic pairing on (the non-divisible quotient of) the \emph{Tate--Shafarevich group} of elliptic curves over number fields. Using arithmetic duality, Tate generalized Cassels's pairing to principally polarized abelian varieties over a global field. He proved that his pairing was always skew-symmetric, and established in \cite[Theorem 3.3]{Tate62} that it is \emph{symplectic}\footnote{We remind the reader that a pairing $\langle -, - \rangle$ is said to be
\[
\text{\emph{skew-symmetric} if } \langle u, v \rangle = -  \langle v,u \rangle \quad \text{ for all } u,v 
\] 
\[
\text{and \emph{alternating} if } \langle u, u \rangle = 0 \quad \text{ for all } u.
\]
Alternating implies skew-symmetric, but the converse can fail if the group has non-trivial 2-torsion. Finally, a pairing is \emph{symplectic} if it is both alternating and non-degenerate.} under a technical hypothesis (which is always satisfied for elliptic curves). In honor of Cassels and Tate, this pairing is now known as the \emph{Cassels--Tate pairing}.

In turn, Cassels seems to have been motivated by a numerology discovered by Selmer, that the Tate--Shafarevich group of an elliptic curve always seems to have size a perfect square \cite[\S 3]{milne2025arithmeticduality}. Because a finite abelian group with a symplectic form must have size a perfect square, this numerology could be explained by the presence of a symplectic structure. 

Thus, the relationship between symplectic structures and perfect squares has been a historical driving force behind the development of modern algebraic number theory. This survey concerns one of the founding questions in this circle of ideas, which turned out to be surprisingly difficult to resolve. 

\subsection{Tate's Symplecticity Conjecture}

A few years after the discovery of arithmetic duality, Tate's 1966 Bourbaki seminar \cite{Tate66} generalized the Birch and Swinnerton-Dyer Conjecture to abelian varieties and formulated its geometric analogue, now called the \emph{Artin--Tate Conjecture}, for a smooth, proper, geometrically connected surface $X$ over a finite field of characteristic $p$. Moreover, Tate conjectured that there should be an analogous symplectic form on the \emph{Brauer group} $\Br(X) := \rH^2_{\et}(X; \G_m)$, a torsion abelian group, conjecturally finite, which is closely connected to the Tate--Shafarevich group.

Let $\Br(X)_{\nd}$ be the non-divisible quotient of $\Br(X)$, i.e., the quotient by the subgroup of divisible elements. The (conjectural) finiteness of $\Br(X)$ would imply that $\Br(X) = \Br(X)_{\nd}$. Artin--Tate constructed a non-degenerate skew-symmetric form on $\Br(X)_{\nd}[\ell^\infty]$ for any $\ell \neq p$, and conjectured in \cite{Tate66} that it is moreover \emph{alternating}. This was named the 
\emph{Artin--Tate pairing} in \cite{Feng20}.

For $\ell = p$, Tate did not define a pairing on $\Br(X)_{\nd}[p^\infty]$, because at the time of writing \cite{Tate66} he lacked a suitable $p$-adic cohomology theory which possessed both the appropriate connection to $\Br(X)$, and a theory of arithmetic duality. Later, Milne provided this theory with his so-called \emph{logarithmic de Rham--Witt cohomology}, which is now 
also known under the synonymous names of \emph{$p$-adic 
\'etale-motivic cohomology}, and \emph{syntomic cohomology}. 
Imitating Artin--Tate's construction, Milne
constructed in \cite{Milne75} a non-degenerate skew-symmetric pairing on $
\Br(X)_{\nd}[p^{\infty}]$ when $p \neq 2$. The restriction $p 
\neq 2$ came from deficiencies in $p$-adic cohomology theory 
for $p=2$ at the time, and was removed in \cite{Milne86} 
using Illusie's development of the de Rham--Witt complex 
\cite{Ill79}. This finally made available a non-degenerate 
skew-symmetric form on all of $\Br(X)_{\nd}$ in all 
characteristics, which we call the \emph{Milne--Artin--Tate 
pairing}. 

We refer to the prediction that the (Milne--)Artin--Tate pairing is symplectic as ``Tate's Symplecticity Conjecture''. 

\begin{conj}[Tate's symplecticity conjecture]\label{conj:symplecticity} 
The Milne--Artin--Tate pairing is alternating. 
\end{conj}

In the next subsection, we discuss some of the broader context motivating Conjecture \ref{conj:symplecticity}. However, the author's own opinion is that the main interest of Conjecture \ref{conj:symplecticity} comes from its contribution to the tradition within number theory of simple empirical observations leading to rich mathematical structures. This is exemplified by Tate's discovery of arithmetic duality as an explanation for why $\# \Sha$ should be a perfect square, and the purpose of this article is to highlight further interesting structures discovered in the journey to the proof of Conjecture \ref{conj:symplecticity}.

\subsection{Broader context} 
Tate's Symplecticity Conjecture arose out of a broader investigation into the Birch and Swinnerton--Dyer Conjecture in positive characteristic, which was the main subject of Tate's Bourbaki report \cite{Tate66}. In this work, Artin--Tate formulated a geometric analogue of the BSD Conjecture, now called the \emph{Artin--Tate Conjecture}, for smooth, proper surfaces $X$ over a finite field. In particular, the Artin--Tate Conjecture predicts that $\# \Br(X)$ is the main invariant featuring in a special value formula for the leading order term of the zeta function of $X$ at the center of its functional equation.

Work of Milne \cite{Milne75, Milne86} and Kato--Trihan \cite{KT} shows that the finiteness of $\Br(X)$, for all such $X$, is equivalent to finiteness of the Tate--Shafarevich group plus the Birch and Swinnerton-Dyer Conjecture for all Jacobian varieties over function fields. If $X$ admits a fibration $f \co X \rightarrow C$ over a curve $C$, then the Artin--Tate Conjecture for $X$ corresponds to the BSD Conjecture for the Jacobian of the generic fiber of $f$, which we denote $J$, over the function field $K$ of $C$. In fact, all the objects featuring in the respective conjectures are related by a dictionary, depicted in Table \ref{tab:Artin--Tate}.

\begin{table}[htbp]
    \centering
    \caption{Analogies between the BSD and Artin--Tate Conjectures.}
    \label{tab:Artin--Tate}
    \begin{tabular}{|c|c|}
        \hline 
        \textbf{BSD Conjecture} & \textbf{Artin--Tate Conjecture} \\
        \hline
        $L(J,s)$ & $\zeta(X,s)$ \\
        \hline
        $J(K)$ & $\mathrm{NS}(X)$\\
        \hline
        $\Sha(J)$ & $\Br(X)$ \\
        \hline
        Cassels--Tate pairing & Artin--Tate pairing \\
        \hline 
    \end{tabular}
\end{table}

Tate's stated motivation for Conjecture \ref{conj:symplecticity} in \cite{Tate66} came from the analogy to $\Sha(J)$. After \cite[Theorem 5.1]{Tate66}, Tate asserts that the Cassels--Tate pairing on $\Sha(J)$ is alternating in general, forgetting the  technical hypothesis from \cite[Theorem 3.3]{Tate62} under which he had proved this alternation. This led to a long misconception that was finally resolved in work of Poonen--Stoll \cite{PS99}, who showed that the Cassels--Tate pairing could fail to be alternating without this assumption.

\subsection{History} 
We survey some prior results on Conjecture \ref{conj:symplecticity}. 

Recall that $\Br(X)_{\nd}$ is finite, so a symplectic form on $\Br(X)_{\nd}$ would imply that $\# \Br(X)_{\nd}$ is a perfect square. Manin \cite{Manin67, Manin86} claimed to find $X$ such that $\# \Br(X) = 2$, which would disprove Conjecture \ref{conj:symplecticity}. However, Urabe later found errors in Manin's computations, and then went on to prove in \cite{Urabe96} that if $p$ (the characteristic of the base field) is odd, then $\# \Br(X)_{\nd}$ is always a perfect square. 

For all $p$, including $p=2$, Liu--Lorenzini--Raynaud proved in \cite{LLR05} that \emph{if} $\Br(X)$ is finite, then its size is a perfect square. The finiteness assumption is very strong: as mentioned above, it is equivalent to the BSD Conjecture for Jacobians over function fields. (The statement that $\#\Br(X)_{\nd}$ is a perfect square when $p=2$ was not known unconditionally until the work \cite{CF} of the author with Carmeli.) Geisser \cite{Gei20} pointed out an error in \cite{LLR05}, stemming from an erroneous Lemma of Gordon \cite{Go79}; this was subsequently fixed in \cite{LLR18}.

Next we turn from the size of the Brauer group to the symplecticity of the (Milne--)Artin--Tate pairing. Zarhin showed in \cite{Zar} that there is an alternating pairing\footnote{Strictly speaking, the result is for a pairing discovered independently in \cite{Zar}, which should be the same as Artin--Tate's, but no comparison with the Artin--Tate pairing is made in \cite{Zar}.} on $\Br(X)_{\nd}$ under the assumptions: 
\begin{itemize}
\item $p \neq 2$, $X$ lifts to characteristic $0$, and the geometric N\'{e}ron-Severi group of $X$ has vanishing 2-primary part.
\end{itemize}
Finally, we arrive at the author's own chapter in the story. The author's Ph.D. thesis \cite{Feng20} established Conjecture \ref{conj:symplecticity} when $p \neq 2$. 

\begin{thm}\label{thm:symplecticityneqp}
Assume $p \neq 2$. Then Conjecture \ref{conj:symplecticity} holds true. 
\end{thm}

The most difficult case of Conjecture \ref{conj:symplecticity} is when $p=2$, because the Conjecture is fundamentally about cohomology at coefficient characteristic $\ell = 2$. When $p=\ell$ the cohomology theory becomes much subtler, as is well documented; the particular technical issues pertinent to this problem will be discussed in \S \ref{sec:ellequalsp}. In \cite{CF}, the author with Shachar Carmeli completed this last case: 

\begin{thm}\label{thm:symplecticityatp}
Assume $p = 2$. Then Conjecture \ref{conj:symplecticity} holds true.  
\end{thm}

\begin{remark}[Higher dimensional generalizations] The works \cite{Feng20} and \cite{CF} construct a generalization of the (Milne--)Artin--Tate pairing for any smooth, proper, geometrically connected variety of even dimension over $\F_p$, and prove that it is symplectic. The case of surfaces, while of historical interest due to the connection to the BSD Conjecture, is not special from this point of view. However, we concentrate on the case of surfaces in this survey, for simplicity. 
\end{remark}

The rest of the paper will survey the ideas going into the proofs of Theorem \ref{thm:symplecticityneqp} and Theorem \ref{thm:symplecticityatp}. We give a brief hint now. The key to Theorem \ref{thm:symplecticityneqp} is artful use of the \emph{\'etale Steenrod operations}, which are subtle symmetries of mod $2$ \'etale cohomology. For $p=2$, a historical difficulty in even \emph{constructing} the pairing at $p=2$ was the lack of an appropriate cohomology theory. The cohomology theory now exists and is called (among other names) syntomic cohomology, but its \emph{symmetries} were not satisfactorily understood. For Theorem \ref{thm:symplecticityatp}, we needed to first build a theory of \emph{syntomic Steenrod operations} acting on syntomic cohomology. This is accomplished following Lurie's vision for ``prismatic stable homotopy''. We emphasize that we develop the theory of syntomic Steenrod operations for all primes $p$, even though the application to Theorem \ref{thm:symplecticityatp} requires only the case $p=2$.

\subsection{Acknowledgments} We thank Shachar Carmeli, Soren Galatius, and Jacob Lurie for their help through the years of working on \cite{Feng20} and \cite{CF}. The author was supported by the NSF (grants DMS-2302520 and DMS-2441922), the Simons Foundation, and the Alfred P. Sloan Foundation.

\section{Cohomology and the (Milne--)Artin--Tate pairing}

\subsection{The dictionary}\label{ssec:cohog-dictionary}

In this article, we will fluidly move between cohomology theories in three different mathematical settings, which have certain common formal properties: 
\begin{itemize}
\item (``topology'') singular cohomology of topological spaces,
\item (``$\ell$-adic'') $\ell$-adic \'etale cohomology of varieties over a field where $\ell$ is invertible. 
\item (``$p$-adic'') syntomic cohomology of varieties over $\F_p$. 
\end{itemize}
Although the definitions and properties of these cohomology theories are established in very different ways in the different situations, the upshot is a formalism that looks uniform. In order to emphasize the parallels, we introduce a uniform (somewhat non-standard) notation for cohomology in the different settings. 
\begin{itemize}
\item (``$\ell$-adic'') If $X$ is a variety over $k$ and $\ell \in k^\times$, then we write $\rH^{a}(X; \Z/\ell^n(b)) := \rH^a_{\et}(X; \mu_{\ell^n}^{\otimes b})$ for the \'etale cohomology of the ``Tate twist'' $\mu_{\ell^n}^{\otimes b}$ viewed as an \'etale sheaf on $X$. We write 
\[
\rH^{a}(X; \Z_\ell(b))  := \limit_n \rH^a_{\et}(X; \Z/\ell^n(b)).
\]
\item (``$p$-adic'') If $X$ is a scheme over $\Z_p$, then we write $\rH^a(X; \Z/p^n(b)) := \rH^{a}_{\syn}(X; \Z/p^n(b))$ for the syntomic cohomology of $X$ with coefficients in the syntomic complex $\Z/p^n(b)$. For smooth schemes $X$ over $\F_p$, this was originally constructed by Milne \cite{Milne75, Milne86} under the name ``logarithmic de Rham--Witt cohomology'', and later reinterpreted by Geisser--Levine \cite{GL00} in terms of mod $p$ motivic cohomology. Recently, a more flexible approach, applicable for general $\Z_p$-schemes, was developed in work of Bhatt--Morrow--Scholze \cite{BMS2} and Bhatt--Lurie \cite{BL22a}; this additional flexibility is used crucially in \cite{CF}. See \cite[\S 2]{CF} for a more detailed discussion and references. 
\item (``topology'') If $X$ is a topological space, then we write $\rH^a(X; \Z/\ell^n(b)) := \rH^a(X; \Z/\ell^n)$. In other words, the twist ``$(b)$'' is a meaningless notational device.
\end{itemize}

\begin{remark}
The uniform notation for these different cohomology theories elides the difficulties behind their construction. Generally speaking, the technological level of the results discussed in this article is 
\begin{itemize}
\item classical in the setting of topology (say, accessible pre-1970), 
\item fairly classical in the setting of $\ell$-adic cohomology (say, accessible pre-1990 with perhaps one technical exception), 
\item modern in the setting of $p$-adic cohomology (\emph{not} accessible pre-2024). 
\end{itemize}
\end{remark}

In order to convey ideas without getting bogged down in technicalities, we will employ the analogies between these settings as an explanatory device. For example, we may ``intuitively justify'' a statement about $\ell$-adic \'etale cohomology by proving an analogous assertion in topology, when we feel that the analogy aptly reflects the moral content. In turn, we will similarly use $\ell$-adic \'etale cohomology as a toy metaphor for syntomic cohomology.

\subsection{Poincar\'e duality}
To illustrate how the dictionary works, we spell out how Poincar\'e duality works in the different settings. 

\subsubsection{Topological incarnation}\label{sssec:top-cohomology}If $M$ is a closed, oriented, connected manifold of dimension $d$, then Poincar\'e duality gives an isomorphism 
\[
\int_M  \co \rH^d(M; \Z/\ell^n) \xrightarrow{\sim} \Z/\ell^n
\]
and a perfect pairing 
\[
\rH^i(M; \Z/\ell^n) \times \rH^{d-i}(M; \Z/\ell^n) \rightarrow \rH^d(M; \Z/\ell^n) \xrightarrow{\int_M} \Z/\ell^n.
\]

\subsubsection{Geometric $\ell$-adic cohomology}\label{sssec:geometric-cohomology}
Next suppose that $X$ is a smooth, projective, connected variety over a \emph{separably closed} field $k$. Let $\ell$ be prime and assume that $k$ does \emph{not} have characteristic $\ell$. Then Poincar\'e duality gives an isomorphism 
\begin{equation}\label{eq:ell-adic-trace-geometric}
\int_X \co \rH^{2d}(X; \Z/\ell^n (d) ) \xrightarrow{\sim} \Z/\ell^n
\end{equation}
and a perfect pairing
\[
\rH^i(X; \Z/\ell^n(b)) \times \rH^{2d-i}(X; \Z/\ell^n(d-b)) \rightarrow \rH^{2d}(X; \Z/\ell^n(d)) \xrightarrow{\int_X} \Z/\ell^n.
\]
This should be thought of as analogous to \S \ref{sssec:top-cohomology}. As a sanity check, when $k = \CC$, the Artin Comparison Theorem identifies the $\ell$-adic \'etale cohomology $\rH^a(X; \Z/\ell^n(b))$ with the singular cohomology $\rH^a(X(\CC); \Z/\ell^n(b))$, compatibly with these forms of Poincar\'e duality. The dimensions match because $X(\CC)$ has dimension $2d$ as a real manifold.

\subsubsection{Arithmetic $\ell$-adic cohomology}\label{sssec:arithmetic-cohomology}
For this paper, we will actually be interested in a different situation, with base field $k$ being a finite field $k = \F_q$ of characteristic $p$. In that case, Poincar\'e duality has a different appearance than in \S \ref{sssec:geometric-cohomology}.

Indeed, we can see by inspection that even the \emph{point} $\Spec \F_q$ has non-trivial cohomology: since $\Gal(\ol \F_q/ \F_q) = \wh \Z$, we have 
\[
\rH^1(\Spec \F_q; \Z/\ell ) = \Hom_{\text{cts}}(\wh \Z, \Z/\ell) = \Z/\ell.
\]
In fact, both the homotopy groups and cohomology groups of $\Spec \F_q$ resemble those of a circle, which is part of the celebrated ``primes and knots'' analogy \cite{Mor24}. It turns out that for general $X$, all the behavioral difference from the situation of \S \ref{sssec:geometric-cohomology} is accounted for by this different behavior of the base ``point''. 

Let $X$ be a smooth, projective, geometrically connected variety over a finite field $\F_q$ of characteristic $p$, with dimension $d$. Then Poincar\'e duality gives an isomorphism 
\begin{equation}\label{eq:ell-adic-trace-absolute}
\int_X \co \rH^{2d+1}(X; \Z/\ell^n (d) ) \xrightarrow{\sim} \Z/\ell^n
\end{equation}
and a perfect pairing
\[
\rH^i(X; \Z/\ell^n(b)) \times \rH^{2d+1-i}(X; \Z/\ell^n(d-b)) \rightarrow \rH^{2d+1}(X; \Z/\ell^n(d)) \xrightarrow{\int_X} \Z/\ell^n.
\]
The ``extra'' $+1$ dimension in \eqref{eq:ell-adic-trace-absolute} comes from the fact that we are really considering the ``absolute'' cohomology $\rH^a(X; \Z/\ell^n(b))$ instead of the ``geometric'' cohomology $\rH^a(X_{\ol \F_q}; \Z/\ell^n(b))$. The two are related by the Hochschild--Serre spectral sequence
\[
\rH^i(\Spec \F_q, \rH^j(X_{\ol k})) \implies \rH^{i+j}(X),
\]
using which one can deduce the ``absolute'' version of Poincar\'e duality from the geometric one.

\subsection{The Brauer group}
Let $X$ be a smooth, projective variety over a finite field $\F_q$ of characteristic $p$. The \emph{Brauer group} of $X$ is $\Br(X):= \rH^2_{\et}(X; \G_m)$. 

For an abelian group $G$, we write $G_{\nd}$ for the maximal non-divisible quotient of $G$, which is the quotient of $G$ by its largest divisible subgroup. 

\begin{lemma}
Let $\ell$ be any prime (possibly equal to $p$). We have an isomorphism 
\[
\Br(X)_{\nd}[\ell^{\infty}] \cong \rH^2(X; \Q_{\ell}/\Z_{\ell}(1))_{\nd}.
\]
\end{lemma}

\begin{proof}
This follows from a formal diagram chase in the cohomology long exact sequence induced by the Kummer sequence of sheaves: see \cite[\S 14.1]{CF}.
\end{proof}

From the short exact sequence
\begin{equation}\label{eq: SES}
0 \rightarrow \Z_{\ell}(d) \rightarrow \Q_{\ell}(d) \rightarrow \Q_{\ell}(1)/\Z_{\ell}(d) \rightarrow 0
\end{equation}
we obtain a boundary map 
\[
\wt{\delta} \co \rH^2(X; \Q_{\ell}/\Z_{\ell}(d)) \rightarrow \rH^3(X; \Z_{\ell}(d)).
\]

\begin{lemma}\label{lem:tors-coh}
Let $\ell$ be any prime (possibly equal to $p$). The map $\wt{\delta}$ induces an isomorphism
\[
\wt{\delta} \co \rH^2(X; \Q_{\ell}/\Z_{\ell}(d))_{\nd} \xrightarrow{\sim} \rH^3(X; \Z_{\ell}(d))_{\tors}.
\]
\end{lemma}

\begin{proof}The long exact sequence associated to \eqref{eq: SES} reads
\[
\begin{tikzcd}[row sep = huge]
\ldots \ar[r] & \rH^2(X; \Q_\ell(d)) \ar[r] & \rH^2(X; \Q_\ell/\Z_\ell(d))  \ar[dll, "\wt \delta"'] \\  
 \rH^{3}(X; \Z_\ell(d)) \ar[r] &  
\rH^{3}(X; \Q_\ell(d))   \ar[r] & \ldots 
\end{tikzcd}
\]
From this, we see that the image of the boundary homomorphism is $\rH^3(X; \Z_\ell(d))_{\tors}$, and its kernel is divisible. Therefore, it factors over an isomorphism $\wt{\delta}$, as claimed. 
\end{proof}

\subsection{The (Milne--)Artin--Tate pairing}
Now let $X$ be a smooth, projective, geometrically connected variety of dimension $2d$ over a finite field $\F_q$ of characteristic $p$. We can now define the Milne--Artin--Tate pairing on $\rH^{2d}(X;  \Q_{\ell}/\Z_{\ell}(d))_{\nd}$ for each prime $\ell$, including $\ell = p$. While historically the cases $\ell \neq p$ and $\ell = p$ should be distinguished, with the former being due to Artin--Tate in \cite{Tate66} and the latter due to Milne in \cite{Milne75, Milne86}, we will refer to all as the ``Milne--Artin--Tate pairing'' in the rest of this article.

\begin{defn}[Milne--Artin--Tate pairing] For $x,x' \in  \rH^{2d}(X;  \Q_{\ell}/\Z_{\ell}(d))_{\nd}$, note that $x \smile \wt \delta x' \in \rH^{4d+1}(X;  \Q_{\ell}/\Z_{\ell}(2d))$. We define 
\[
\langle x, x' \rangle_{\mrm{MAT}} := \int_X x \smile (\wt \delta x') \in \Q_\ell/\Z_\ell.
\]
\end{defn}

\begin{remark}[Linking form]\label{rem:linking-form}
In the topological setting, this definition is well-known, and is called the ``linking form'' for an odd-dimensional manifold. For an even-dimensional manifold $M^{2d}$, there is a well-known intersection form on $\rH_{d}(M^{2d}; \Z)_{\mrm{free}}$. For an odd-dimensional manifold $M^{2d+1}$, there is no ``middle'' (co)homology, and there is instead the linking form on $\rH_d(M^{2d+1}; \Z)_{\tors} \cong \rH^{d+1}(M^{2d+1}; \Z)_{\tors}$. It can be described geometrically as follows. Let $u, v \in \rH_d(M^{2d+1}; \Z)_{\tors}$. Since $v$ is torsion, there is an integer $n \geq 1$ such that $n v$ bounds a $(d+1)$-cycle $V$. Intersecting $V$ with a cycle representing $u$, we obtain an integer $u \cdot V \in \Z$. But we need to do some accounting to make this well-defined. First, we need to divide by $n$ to renormalize for that choice, obtaining an element of $\Q$. Second, the choice of $V$ was itself only well-defined up to cycles, so the resulting rational number is only well-defined up to translation by an integer. Hence the linking form is ultimately valued in $\Q/\Z$. 
\end{remark}

Next we define an auxiliary pairing on the group $\rH^{2d}(X; \Z/\ell^n\Z(d))_{\nd}$ which will be a useful technical device to access the Milne--Artin--Tate pairing. The short exact sequence of sheaves
\[
0 \rightarrow \Z/\ell^n\Z(d) \xrightarrow{\ell^n}  \Z/\ell^{2n} \Z(d) \rightarrow \Z/\ell^n\Z(d) \rightarrow 0
\]
inducing the \emph{Bockstein operation}
\begin{equation}\label{eq:bockstein}
\beta \colon \rH^i(X; \Z/\ell^n\Z(d)) \rightarrow \rH^{i+1}(X; \Z/\ell^n\Z(d)).
\end{equation}

\begin{defn}
We define the pairing 
\[
\langle - , - \rangle_{n} \colon \rH^{2d}(X; \Z/\ell^n\Z(d)) \times \rH^{2d}(X; \Z/\ell^n\Z(d)) \rightarrow \Z/\ell^n\Z
\]
by 
\[
\langle x, y \rangle_n := \int_X x \smile \beta y .
\]
(Note that $\langle - , - \rangle_n$ need not be non-degenerate.)
\end{defn}

What is the relation between $\langle \cdot , \cdot \rangle_{n}$ and the Artin--Tate pairing? The long exact sequence associated to 
\[
0 \rightarrow \Z_{\ell}(d) \rightarrow \Z_{\ell}(d) \rightarrow \Z/\ell^n\Z(d) \rightarrow 0 
\]
provides a surjection
\begin{equation}\label{eq:torsion-bockstein}
\rH^{2d}(X; \Z/\ell^n\Z(d)) \twoheadrightarrow \rH^{2d+1}(X; \Z_\ell(d))[\ell^n].
\end{equation}

\begin{prop}\label{prop: compatibility} The map \eqref{eq:torsion-bockstein} intertwines the pairings $\langle \cdot , \cdot \rangle_n$ and $\langle - , - \rangle_{\mrm{MAT}}$, in the sense that the following diagram commutes:
\[
\xymatrix @C=0pc{
\rH^{2d}(X; \Z/\ell^n\Z(d)) \ar@{->>}[d] & \times &  \rH^{2d}(X; \Z/\ell^n\Z(d))  \ar@{->>}[d]  \ar[rrrrrrr]^{\langle \cdot , \cdot \rangle_n} &&&&&&& \rH^{4d+1}(X; \Z/\ell^n\Z(2d)) \ar[d]_{\sim}  \\
\rH^{2d+1}(X; \Z_\ell(d))[\ell^n] & \times & \rH^{2d+1}(X; \Z_\ell(d))[\ell^n]  \ar[rrrrrrr]_{\langle \cdot , \cdot \rangle_{\mrm{MAT}}} &&&&&&& \rH^{4d+1}(X; \Q_\ell/\Z_\ell(2d)) [\ell^n]
	}
\]
\end{prop}

\begin{proof}
This follows formally from a diagram chase: see \cite[Proposition 2.5]{Feng20}.
\end{proof}

Thanks to Proposition \ref{prop: compatibility}, in order to prove Theorem \ref{thm:symplecticityatp} and Theorem \ref{thm:symplecticityneqp} it suffices to show that 
$\langle - , - \rangle_n$ is alternating, which is what we will do. What makes $\langle - , - \rangle_n$ better to work with than $\langle - , - \rangle_{\mrm{MAT}}$? Philosophically this seems to come down to the fact that $\Z/\ell^n \Z$ enjoys a \emph{ring structure} while $\Q_{\ell}/\Z_{\ell}$ does not. The ring structure is (partly) responsible for extra structure on $\rH^*(X; \Z/\ell^n \Z)$ that will play a crucial role in our argument. To give a sample illustration of this, we will show how it can be used to show skew-symmetry. 

\begin{prop}\label{prop: skew-symmetric}
For every prime $\ell$, and every $n \geq 1$, the pairing $\langle - , - \rangle_n$ is skew-symmetric. 
\end{prop} 

\begin{proof}
The assertion is equivalent to 
\[
x \smile \beta y  + y \smile \beta  x  = 0.
\]
Since $\beta$ is a derivation, we have $ x \smile \beta y  + y \smile \beta  x  = \beta(x \smile y)$. But the boundary map 
\[
\beta  \colon \rH^{4d}(X; \Z/\ell^n\Z(2d)) \rightarrow  \rH^{4d+1}(X; \Z/\ell^n\Z(2d))
\]
vanishes because its image is the kernel of 
\[
[\ell^{n}] \co \rH^{4d+1}(X; \Z/\ell^n\Z(2d)) \rightarrow \rH^{4d+1}(X; \Z/\ell^{2n}\Z(2d))
\]
which is identified with the inclusion $\ell^n \Z/\ell^{2n}\Z \hookrightarrow \Z/\ell^{2n}\Z$ by Poincar\'{e} duality. 
\end{proof}

\begin{cor}
The Milne--Artin--Tate pairing is skew-symmetric.
\end{cor}

\subsection{Toy example: quadratic reciprocity}\label{ssec:QR} Thus far, the discussion has been quite formal (given arithmetic duality). Nevertheless, it encodes interesting arithmetic information. As a toy illustration, we recast quadratic reciprocity in terms of arithmetic duality.

For two disjoint knots $L,K \subset S^3$, there is a \emph{linking number} $\lk(L,K) \in \Z$ that measures ``how many times (counted with multiplicity) $K$ winds around $L$''. This can be defined algebraically by noting that $\rH_1(S^3 \setminus K; \Z) \cong \Z$. The knot $K$ defines a class in $\rH_1(S^3 \setminus L; \Z)$, whose image in $\Z$ under this identification is the linking number.
 
\begin{figure}[htbp]
    \centering
    \includegraphics[scale=.15]{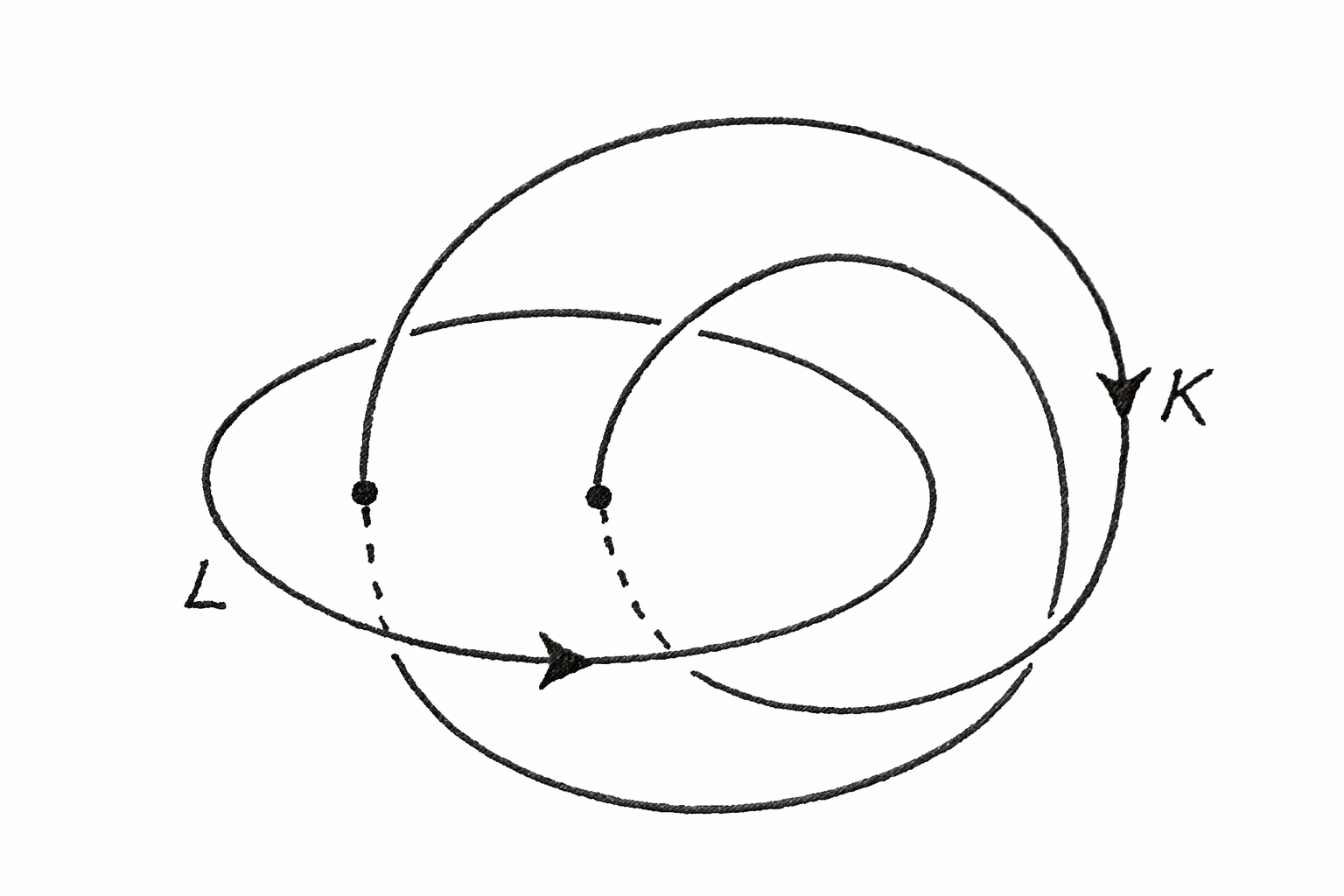}
    \caption{$\lk(L, K)=2$}
    \label{fig:link2}
\end{figure}
 There is an analogue of the linking number for primes. We just give the case of odd primes for simplicity: if $p,q$ are odd primes, then 
\begin{equation}\label{eq:prime-link}
\lk(p,q) \pmod{2} = \begin{cases} 0 & (-1)^{\epsilon(p)} p \text{ is a square mod $q$} \\ 1 & \text{otherwise} \end{cases} \in \Z/2,
\end{equation}
where $\epsilon(p) = 0$ if $p \equiv 1 \pmod{4}$ and $\epsilon(p) = 1$ if $p \equiv 3 \pmod{4}$. To see why this definition is analogous to a linking number, we recast the linking number in algebraic terms. One calculates that $\rH_1(S^3 \setminus K; \Z) \cong \Z$, which implies that there is a unique $\Z/2\Z$-cover $T \rightarrow S^3 \setminus K$. Then $\lk(L,K) \equiv 0 \pmod{2}$ if and only if $L$ splits in $T$. 

Analogously, there is a unique \'etale double cover $\Spec \cO \rightarrow \Spec \Z[1/p]$, for which $\cO = \Z[1/p][\frac{1+\sqrt{(-1)^{\epsilon(p)} p}}{2}]$. Then (using that $q$ is odd) $q$ splits in $\cO$ if and only if $(-1)^{\epsilon(p)} p$ is already a perfect square mod $q$. This justifies why \eqref{eq:prime-link} is analogous to the mod $2$ linking number. 

Now, it is visually clear that the linking number of knots is symmetric: $\lk(L, K) = \lk(K,L)$. See \cite[Problem 13]{Mil97} for an outline of the proof. Much less intuitive is the assertion that $\lk(p,q)= \lk(q,p)$ for distinct odd primes $p$ and $q$---this is the famous \emph{Law of Quadratic Reciprocity} proved by Gauss! 

This example demonstrates that simple topological dualities can be parallel to deep arithmetic ones under our dictionary. 
To be clear, this is not a proof of quadratic reciprocity; class field theory is already built into Tate's arithmetic duality theorems. The example should be thought of as presenting arithmetic duality as a \emph{worldview}, which makes quadratic reciprocity seem natural.

\section{The connection to Steenrod operations}\label{sec:steenrod} One of the most interesting facets of this story is the connection between Conjecture \ref{conj:symplecticity} and Steenrod operations. We will attempt to convey how Steenrod operations enter the narrative. 

\subsection{What are Steenrod operations?}\label{ssec:Steenrod} Steenrod operations were discovered by Norman Steenrod \cite{Ste47} in the context of singular cohomology of topological spaces, so we begin our story in that context. 

Let $T$ be a topological space. Then $\rH^*(T; \Z/2)$ is a commutative ring over $\F_2$, and as such has a Frobenius endomorphism $x \mapsto x^2$. However, it turns out that because $\rH^*(T; \Z/2)$ arises in a special way, it automatically inherits the richer structure of a module over the \emph{Steenrod algebra}. Concretely, this amounts to the existence of natural transformations 
\[
\Sq^i \co \rH^*(T; \Z/2) \rightarrow \rH^{*+i}(T; \Z/2),
\]
called ``Steenrod squares'' or ``Steenrod operations'', satisfying certain relations. It may be helpful to think of $\Sq^j$ as being ``derived'' endomorphisms of the Frobenius endomorphism, although we emphasize that this is not meant in any precise technical sense. On $\rH^j(T; \Z/2)$, the operation $\Sq^j$ is the Frobenius (i.e., squaring) operation $x \mapsto x^2$, and $\Sq^i$ vanishes for $i>j$. Thus, the indexing is that $\Sq^{j-1}$ is the first ``derived'' operation. Let us explain what it is, following the presentation of \cite[Chapter 2]{MoTan68}.

\subsubsection{The first Steenrod operation} 
As is well-known, the cup product on cohomology is graded commutative. The cup product is induced by a chain-level operation, which however is \emph{not} graded-commutative: at the chain level, we only have commutativity ``up to homotopy''. Steenrod expressed this systematically in \cite{Ste47} by constructing a ``cup-$1$ product'' 
\begin{align*}
\mrm{cup}_1 \co C^r(T;\Z) \otimes C^s(T;\Z) & \rightarrow C^{r+s-1}(T;\Z) \\
u \otimes v  & \mapsto u \smile_1 v 
\end{align*}
such that for all $u \otimes v \in C^r(T;\Z) \otimes C^s(T;\Z)$,  we have 
\begin{equation}\label{eq:cup-1-product}
d (u  \smile_1 v) = -d  u \smile_1 v + (-1)^{r+1} u  \smile_1 d v +  u \smile  v - (-1)^{rs} v \smile u.
\end{equation}
If we take $u = v \in Z^r(T; \Z)$ to be a cocycle, then this relation specializes to 
\[
d(u \smile_1 u) = u \smile u - (-1)^{r^2} u \smile u
\]
which implies that $u \smile_1 u$ is itself a cocycle with coefficients modulo $2$. The map $u \mapsto u \smile_1 u$ descends to the operation 
\[
\Sq^{r-1} \co \rH^r(T; \Z/2) \rightarrow \rH^{2r-1}(T; \Z/2).
\]

\subsubsection{All Steenrod operations} 
More generally, for each $i=1, 2, \ldots$ Steenrod defined ``cup-$i$ products''
\begin{align*}
\mrm{cup}_i \co C^r(T;\Z) \otimes C^s(T;\Z) &\rightarrow C^{r+s-i}(T;\Z) \\
u \otimes v & \mapsto u \smile_i v 
\end{align*}
such that for all $u \otimes v \in C^r(T;\Z) \otimes C^s(T;\Z)$, we have 
\begin{equation}\label{eq: homotopy formula}
d(u \smile_i v) - (-1)^i  \mrm{cup}_i(d(u \otimes v)) = (-1)^{i-1} u \smile_{i-1} v - (-1)^{rs} v \smile_{i-1}u. 
\end{equation}
Taking $u=v \in Z^r(T; \Z)$, we find that the operation $u \mapsto u \smile_i u$ descends to a cohomology operation 
\[
\Sq^{r-i} \co \rH^r(T; \Z/2) \rightarrow \rH^{2r-i}(T; \Z/2).
\]

\begin{remark}
The Frobenius endomorphism of $\rH^*(T; \Z/2)$ can be thought of as specializing the multiplication, a priori a function of two variables, to diagonal inputs. Analogously, the Steenrod operations can be thought of as specializing the higher $\mrm{cup}_i$-products, a priori functions of two variables, to diagonal inputs. This is the sense in which they can be thought of as ``derived'' endomorphisms of Frobenius. 
\end{remark}

\subsubsection{Generalized coefficients}\label{sssec:generalized-Steenrod}
The Steenrod squares are used to analyze the (Milne--)Artin--Tate pairing on $\Br(X)[2]$. If one wants to analyze the pairing on $\Br(X)[2^n]$, then one needs ``generalized Steenrod squares'' defined in \cite[\S 3.5]{Feng20} (and surely well-known to experts, although we did not find discussion of them in the literature) of the form 
\[
\wt \Sq^i \co \rH^r(T; \Z/2^n) \rightarrow \rH^{r+i}(T; \Z/2^n). 
\]
For simplicity, we shall omit this part of the story in this article.

\subsection{$\EE_\infty$ Steenrod operations}\label{ssec:E-infty-steenrod}
We just discussed Steenrod operations in their original setting of algebraic topology. 
The nature of that discussion already hinted that the construction was not fundamentally tied to topological spaces, and should have a more general explanation in homotopical algebra. This is indeed the case, which is fortunate because we actually need the Steenrod operations in a different setting: the cohomology of algebraic varieties. 

In general, the ring of cochains of any site with coefficients in a commutative ring lifts to an ``$\EE_\infty$-algebra'', which equips the cohomology ring with Steenrod operations \cite{May70}. In particular, this applies to both the $\ell$-adic \'etale cohomology studied in \cite{Feng20}, and the syntomic cohomology studied in \cite{CF}. We call these ``\emph{$\EE_\infty$ Steenrod operations}'' in \cite{CF}, to distinguish them from another type of operation which is related to Voevodsky's \emph{motivic Steenrod algebra} (which we have not yet discussed). The $\EE_\infty$-operations are the ones used in \cite{Feng20}, although there it was technically more convenient to follow a different construction using \'etale homotopy theory, and thus bootstrap the operations directly from topology. 

In the $p$-adic setting, to deal with the twists, we should package coefficients in the following way, because we want coefficients to have a \emph{ring structure}. We have maps 
\[
\Z/p(b) \otimes \Z/p(b') \rightarrow \Z/p(b+b')
\]
which assemble into a commutative ring structure on
\[
\bigoplus_{b \in \Z} \Z/p (b).
\]
This equips 
\[
\bigoplus_{b \in \Z} \rH^*(X; \Z/p(b))
\]
with $\EE_\infty$ Steenrod operations. For $p = 2$, each $\Sq^i$ has the effect of doubling the weight $b$. In the $\ell$-adic settings, it is arguably most canonical to track this grading, although its effect can be trivialized using the canonical identifications $\Z/2 = \Z/2(1) = \Z/2(2) = \ldots$.

\subsection{Bockstein operations} The connection between Steenrod operations and the (Milne--)Artin--Tate pairing comes from the following key formula. 

\begin{thm}\label{thm: MAT form}
Let $X$ be one of the following objects: 
\begin{itemize}
\item a connected, closed, oriented manifold of dimension $4d+1$, or 
\item a smooth, proper, geometrically connected variety of dimension $2d$ over a finite field. 
\end{itemize}
Then for all $u \in \rH^{2d}(X; \Z/2^n(d))$, we have
\begin{equation}\label{eq:bock-identity}
u \smile \beta_n (u) = [2^{n-1}] \circ \Sq^{2d}(\ol{\beta_n (u)})
\end{equation}
where: 
\begin{itemize}
\item $\beta_n \co \rH^{2d}(X; \Z/2^n(d)) \rightarrow \rH^{2d+1}(X; \Z/2^n(d))$ is the Bockstein homomorphism from \eqref{eq:bockstein},
\item $\ol{\beta_n (u)}$ denotes the reduction of $\beta_n (u)$ mod $2$. 
\item $[2^{n-1}] \co \rH^{4d+1}(X; \Z/2(2d)) \rightarrow \rH^{4d+1}(X; \Z/2^n(2d))$ is the map on cohomology induced by the map of coefficients $\Z/2(2d)  \rightarrow \Z/2^n(2d)$ given by multiplication by $2^{n-1}$. 
\end{itemize}
\end{thm}

We caution that the identity \eqref{eq:bock-identity} makes sense in large generality, for example without assuming that $X$ is smooth or proper or even of dimension $2d$, but is not true without assumptions. What is true is that in general, the difference between the left and right sides in Theorem \ref{thm: MAT form} occurs naturally as a differential of some element in a spectral sequence, the \emph{Bockstein spectral sequence}, which ``converges from the mod $2$ cohomology to integral cohomology''. The Theorem is then proved by calculating that this differential vanishes for other reasons, which in this case ultimately comes from information about the integral cohomology that we have thanks to Poincar\'{e} duality. 

\begin{proof}[Hint of proof]
We will give a proof of \eqref{eq:bock-identity} for a special case: we assume $n=1$, and we consider only the topological setting. To emphasize this, we write $T := X$. This case is enough to capture the core algebra, in particular illuminating the entrance of Steenrod operations, while avoiding nontrivial but orthogonal technical issues. 

In this case, the proof has the following format. We will produce a sequence of cohomology operations $\beta_1, \beta_2 $ on  $\rH^*(T; \Z/2\Z)$ such that $\beta_2$ is defined on the kernel of $\beta_1$ and takes values in the cokernel of $\beta_1$. The operation $\beta_1 = \beta$ is the Bockstein operation of \eqref{eq:bockstein}. It is described by the following recipe: 
\begin{enumerate}
\item Lift $x \in \rH^*(T; \Z/2\Z)$ to an integral cochain $\wt{x} \in C^*(T; \Z)$. 
\item Since $x$ came from a cocycle mod $2$, we have $d \wt{x} = 2 \wt{y}$ for some cocycle $\wt{y} \in Z^*(T; \Z)$. 
\item Output $\beta(x) = y$, the cohomology class of $\wt{y} \pmod{2}$.
\end{enumerate}
Next we define the ``secondary'' operation $\beta_2$ on the kernel of $\beta$. If $\beta(x) =0$, this means that the lifts in the recipe above could have been chosen so that $d \wt{x} = 4 \wt{y}'$ was actually divisible by $4$. Then set $\beta_2(x)$ to be the cohomology class of $\wt{y}' \pmod{2}$. When tracking the choices involved, one sees $\beta_2$ is well-defined up to the image of $\beta_1$. 

Consider $\beta_1$ and $\beta_2$ applied to $x = u^2$ as in Theorem \ref{thm: MAT form}. Then it is clear that 
\[
\beta(x) = 2 u \smile \beta(u) = 0 \in \rH^{4d+1}(T; \Z/2\Z(2d)).
\]
To find a suitable $\wt{y}'$, we need to find an integral cochain lift of $u^2$ whose coboundary is divisible by $4$. We start by picking an arbitrary lift $\wt{u} \in C^{2d}(T; \Z)$ and then squaring it. Since $\wt u$ is a cocycle mod $2$, it follows that $d \wt{u} = 2 \wt v$ for some $\wt v \in C^{2d+1}(T; \Z)$, whose mod $2$ cohomology class is $\beta(u)$. Then we find using \eqref{eq:cup-1-product} that
\begin{align*}
d(\wt{u}^2   ) &=  2\wt{v} \smile  \wt{u} +  2 \wt{u} \smile \wt{v} \\
& = 2(2\wt{v} \smile \wt{u}  + [\wt{u} ,  \wt{v} ]) \\
&= 4 \wt{v}\smile \wt{u} +  2 (d \mrm{cup}_1(\wt{v} \otimes \wt{u}) + \mrm{cup}_1 d(\wt{v} \otimes \wt{u}) ) \\
&= 4 \wt{u}\smile \wt{v} +  d2 \mrm{cup}_1(\wt{v} \otimes \wt{u}) + 4\mrm{cup}_1 (\wt{v}  \otimes \wt{v}) .
\end{align*}
Rearranging this to say 
\[
d(\wt{u}^2-  2\mrm{cup}_1(\wt{u} \otimes \wt{v}) ) = 4 (\wt{u}\smile \wt{v} +   \mrm{cup}_1(\wt{v}  \otimes \wt{v}) ),
\]
we have successfully produced a lift $\wt{x} := \wt{u}^2-  2\mrm{cup}_1(\wt{u} \otimes \wt{v})$ whose coboundary is $4$ times $\wt{y}' = \wt{u} \smile \wt{v} +   \mrm{cup}_1(\wt{v}  \otimes \wt{v})$. By definition, the cohomology class of the mod $2$ reduction of $\wt{y}'$ is $u \smile \beta(u) + \Sq^{2d}(\beta (u))$. 

Note that $\wt{u} \smile \wt{v}$ and $\wt{y}'$ visibly represent torsion classes in $\rH^{4d+1}(T; \Z)$. Since this group is torsion-free by Poincar\'{e} duality, the images of $\beta_1$ and $\beta_2$ must both vanish in $\rH^{4d+1}(T; \Z/2)$. This shows that $ u \smile \beta(u) + \Sq^{2d}(\beta(u)) = 0 \in \rH^{4d+1}(T; \Z/2)$, as desired. 

\end{proof}

\subsection{Overview of the rest of the article} Recall from Remark \ref{rem:linking-form} that the linking form is the topological analogue of the Milne--Artin--Tate pairing. Our proofs of Theorem \ref{thm:symplecticityneqp} and Theorem \ref{thm:symplecticityatp} can be organized into several layers, depending on how setting-specific they are (see Figure \ref{fig:proof-layers}). 

Some of the initial layers are universal enough to apply equally well to the linking form, and they can be used to give a necessary and sufficient criterion for the linking form to be alternating. Indeed, the connection to Steenrod operations in Theorem \ref{thm: MAT form} is of this nature.

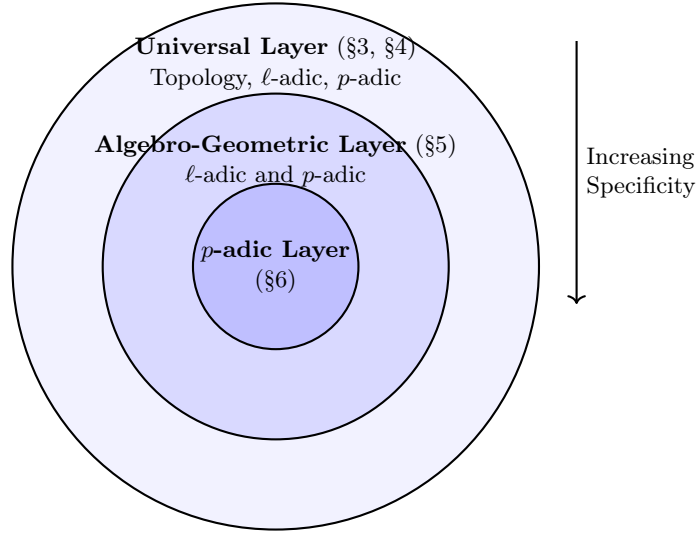
\begin{figure}[h]
    \centering
    \begin{tikzpicture}[node distance=2cm, font=\small]
        \draw[thick, fill=blue!5] (0,0) circle (3.5cm);
        \draw[thick, fill=blue!15] (0,0) circle (2.3cm);
        \draw[thick, fill=blue!25] (0,0) circle (1.1cm);

        \node at (0, 2.9) {\textbf{Universal Layer} (\S \ref{sec:steenrod}, \S \ref{sec:linking-form})};
        \node at (0, 2.5) {Topology, $\ell$-adic, $p$-adic};
        
        \node at (0, 1.6) {\textbf{Algebro-Geometric Layer} (\S \ref{sec:ell-adic})};
        \node at (0, 1.2) {$\ell$-adic and $p$-adic};
        
        \node at (0, 0.2) {\textbf{$p$-adic Layer}};
        \node at (0, -0.2) {(\S \ref{sec:ellequalsp})};

        \draw[->, thick] (4, 3) -- (4, -0.5) node[midway, right, align=left] {Increasing\\Specificity};
        
    \end{tikzpicture}
    \caption{Organization of the remaining sections of this paper.}
    \label{fig:proof-layers}
\end{figure}

Over the rest of the article, we will present the layers in order from the most general to the most specific. In \S \ref{sec:linking-form}, we study the symplecticity of the linking form using ideas which are common to all three settings (``topology'', ``$\ell$-adic'', and ``$p$-adic''). Then in \S \ref{sec:ell-adic}, we sketch ideas going into the proof of Theorem \ref{thm:symplecticityneqp} which are special to the algebro-geometric (``$\ell$-adic'' and ``$p$-adic'') settings. Finally, in \S \ref{sec:ellequalsp} we sketch ideas going into the proof of Theorem \ref{thm:symplecticityatp}, which are furthermore special to the $p$-adic setting.

\section{Analysis of the linking form}\label{sec:linking-form}

The topological analogue of the Milne--Artin--Tate pairing on a variety of dimension $2d$ is the linking form (Remark \ref{rem:linking-form}) on an orientable, closed, connected manifold $M$ of dimension $4d+1$. The question of when this linking form is alternating can be viewed as a toy metaphor for Conjecture \ref{conj:symplecticity}, and we will study it in this section.

\subsection{Criterion for alternation}\label{ssec:linking-form-alternation}

Motivated by Theorem \ref{thm: MAT form}, we want to be able to compute Steenrod operations such as 
\[
\Sq^{2d} \co \rH^{2d+1}(M; \Z/2(d)) \rightarrow \rH^{4d+1}(M; \Z/2(2d)).
\]
One observation is that this map must, by Poincar\'{e} duality, be represented by a unique class $v_{2d} \in \rH^{2d}(M; \Z/2(d))$, in the sense that 
\[
\Sq^{2d} (x) = v_{2d} \smile x \text{ for all } x \in \rH^{2d+1}(M; \Z/2(d)).
\]
We can then give a necessary and sufficient criterion for the linking form to be alternating in terms of this class $v_{2d}$. 

\begin{thm}[{\cite[Theorem 7.4]{Feng20}}]\label{thm:alternating-linking-form}
Let $M$ be an orientable, closed, connected manifold of dimension $4d+1$. Then the linking form on $\rH^{2d+1}(M; \Z)_{\tors}$ is alternating if and only if $v_{2d}$ lifts to $\rH^{2d}(M; \Z)$, i.e., if and only if $\wt{\beta} (v_{2d}) = 0$ where $\wt{\beta}$ is the boundary map for the short exact sequence
\[
0 \rightarrow \Z \xrightarrow{2} \Z \rightarrow \Z/2 \Z \rightarrow 0.
\]
\end{thm}

\begin{proof}[Proof sketch] Writing 
\[
\rH^{2d+1}(M; \Z)_{\tors} = \bigoplus_\ell \rH^{2d+1}(M; \Z)[\ell^\infty],
\]
it suffices to examine $\ell=2$, since we already showed in Proposition \ref{prop: skew-symmetric} that the form is skew-symmetric. For simplicity, we will content ourselves with proving that the pairing is alternating on $\rH^{2d+1}(M; \Z)[2]$ if and only if $v_{2d}$ lifts to $\rH^{2d}(M; \Z/4(d))$. The argument contains all the key features of the general case, but the general case requires the ``generalized Steenrod operations'' of \S \ref{sssec:generalized-Steenrod}, and the full power of Theorem \ref{thm: MAT form}, where we only proved the $n=1$ case earlier. 

By Proposition \ref{prop: compatibility}, it is equivalent to prove that the pairing $\langle -, - \rangle_1$ on $\rH^{2d}(M; \Z/2(d))$ is alternating if and only if $v_{2d}$ lifts to $\rH^{2d}(M; \Z/4(d))$. Using Theorem \ref{thm: MAT form}, we have for all $u \in \rH^{2d}(M; \Z/2(d))$ that
\[
u \smile \beta(u) = \Sq^{2d}(\beta (u)).
\]
We want to identify when this vanishes for all $u \in \rH^{2d}(M; \Z/2(d))$. By definition of $v_{2d}$, we have $\Sq^{2d} (\beta (u)) = v_{2d} \smile  \beta (u)$. So it is equivalent to identify when $0 = v_{2d} \smile \beta (u) $ for all $u$. Since $\beta$ is a derivation, we can rewrite
\[
v_{2d} \smile \beta (u) = \beta( v_{2d} \smile u) - \beta(v_{2d}) \smile u.
\]
As discussed in the proof of Proposition \ref{prop: skew-symmetric}, $\beta$ vanishes on $\rH^{4d}(M; \Z /2(2d))$, so $\beta( v_{2d} \smile u)  =0$. The question is then when $\beta(v_{2d}) = 0$. By the long exact sequence this is the obstruction to lifting $v_{2d}$ to $\rH^{2d}(M; \Z/4(d))$, as desired. 
\end{proof}

\begin{example}\label{ex:wu-manifold}
The hypothesis of Theorem \ref{thm:alternating-linking-form} does not always hold. For example, the ``Wu manifold'' $\mrm{SU}(3)/\mrm{SO}(3)$ is a counterexample for $d=1$. 
\end{example}

By itself, Theorem \ref{thm:alternating-linking-form} is not necessarily that useful without a way to check the condition on $v_{2d}$, or at least to relate $v_{2d}$ to other more familiar characteristic classes. This is the direction that we will explore next. 

\subsection{Connection to Wall's Theorem}\label{ssec:wall-5-manifold} For the special case of simply connected, closed, smooth 5-manifolds (which are cohomologically analogous to a surface over a finite field), Wall \cite{Wall62} had already identified the criterion for the alternation of the linking form. To formulate it, we define the map 
\begin{align*}
\rH_2(M; \Z)_{\tors} & \rightarrow \Z/2 \\
x & \mapsto w_2(x)
\end{align*}
as the composition 
\[
\rH_2(M; \Z)_{\tors} \cong \rH^3(M; \Z)_{\tors} \xrightarrow{\smile w_2} \rH^5(M; \Z/2) \cong \Z/2
\]
where $w_2 = w_2(TM) \in \rH^2(M; \Z/2)$ is the second Stiefel--Whitney class of the tangent bundle of $M$. In \cite[Propositions 1 and 2]{Wall62}, Wall proved the following result. 

\begin{thm}[Wall]
Let $M$ be a simply connected, closed, smooth 5-manifold. Then the linking form on $M$ is alternating if and only if $w_2(x) = 0$ for all $x \in \rH_2(M; \Z)_{\tors}$.
\end{thm}

Wall's proof is specific to 5-manifolds, and uses the smooth structure of $M$ (indeed, $w_2$ is defined in terms of this smooth structure). It turns out that for an orientable smooth manifold, in particular a simply connected smooth manifold, the ``Wu class'' $v_2$ from Theorem \ref{thm:alternating-linking-form} agrees with $w_2$. This is a special case of a much more general statement called ``Wu's formula'', which we discuss shortly below in \S \ref{ssec:wu-formula}. Assuming it for now, we will deduce Wall's Theorem from Theorem \ref{thm:alternating-linking-form}. 

\begin{proof}
It suffices to show that $x \smile w_2 = 0$ for all $x \in \rH^3(M; \Z)[2^\infty] \cong \rH_2(M; \Z)[2^\infty]$ if and only if $w_2$ lifts to $\rH^2(M; \Z)$. If $x \in \rH^3(M; \Z)$ is killed by $2^n$, then $x = \beta_n(u)$ for some $u \in \rH^2(M; \Z/2^n)$, where $\beta_n$ is the Bockstein connecting map for the short exact sequence of sheaves
\[
0 \rightarrow \Z \xrightarrow{2^n} \Z \rightarrow \Z/2^n \rightarrow 0.
\]
By the same argument as in the proof of Proposition \ref{prop: skew-symmetric}, we have 
\begin{equation}\label{eq:wall-pairing}
\beta_n(u) \smile w_2 = \beta_n(u \smile w_2) +  u  \smile \beta_n(w_2) = u \smile \beta_n(w_2).
\end{equation}
Hence if $\beta_n(w_2)$ vanishes, then \eqref{eq:wall-pairing} is clearly zero. Conversely, since the pairing between $\rH^2(M; \Z/2^n)$ and $\rH^3(M; \Z/2^n)$ is perfect, if \eqref{eq:wall-pairing} vanishes for all $u$ then $\beta_n(w_2) = 0$. If this vanishes for all $n$, then $w_2$ even vanishes under the Bockstein to $\rH^3(M; \Z)$ coming from the short exact sequence $\Z \rightarrow \Z \xrightarrow{2} \Z/2$, hence lifts to $\rH^2(M; \Z)$. 
\end{proof}

\subsection{Wu's formula}\label{ssec:wu-formula}
We can generalize the construction of the $v_{2d}$ from Theorem \ref{thm:alternating-linking-form}. Let $M$ be a closed, connected topological manifold of dimension $d$. By Poincar\'{e} duality, the operation
\[
\Sq^i \co \rH^{d-i}(M; \Z/2\Z) \rightarrow \rH^{d}(M; \Z/2\Z) \cong \Z/2\Z
\]
must be given by cupping with a (unique class)  $v_i \in \rH^i(M; \Z/2\Z)$. Wu \cite{Wu50} proved the following description of the $v_i$. 

\begin{thm}[Wu's formula]\label{thm:wu}
Let $M$ be a closed smooth manifold. Write 
\begin{align*}
\Sq& := \Sq^0 + \Sq^1 + \Sq^2 + \ldots \\
v &:= v_0 + v_1 + v_2 + \ldots \\
w &:= w_0 + w_1+ w_2 + \ldots
\end{align*}
where the $w_i$ are the Stiefel-Whitney classes of $TM$. Then we have $\Sq(v) = w$. 
\end{thm}

\begin{example}\label{ex:wu-formula}
We can unpack Theorem \ref{thm:wu} as
\[
(1+\Sq^1 + \Sq^2  + \ldots) (1+v_1 + v_2 + \ldots ) = 1+ w_1 + w_2 + \ldots
\]
and then match terms by cohomological degree on each side to solve for $v_i$ for $w_{\leq i}$. For example, using that $\Sq^i$ vanishes on $\rH^{<i}$, we find that $v_1 = w_1$ and $v_2 + \Sq^1(v_1)  = w_2 $. Since $\Sq^i$ agrees with the squaring map on $\rH^i$, this simplifies to $v_2 = w_2 + w_1^2$. If $M$ is orientable, then $w_1 = 0$ and we find that $v_2 = w_2$, as claimed in \S \ref{ssec:wall-5-manifold}.
 \end{example}

\begin{remark}\label{rem:spin}
Suppose $M$ is a closed, orientable, connected smooth manifold of dimension $5$. Putting Theorem \ref{thm:alternating-linking-form} together with Example \ref{ex:wu-formula}, we see that the linking form on $M$ is alternating if and only if $w_2$ lifts to $\rH^2(M; \Z)$. This latter condition is referred to as $M$ being ``spin$^c$''. 
\end{remark} 

To prepare for the proof of Theorem \ref{thm:wu}, we need to recall a perspective on Stiefel--Whitney classes introduced by Thom in \cite{Thom51}. Let $M$ be a topological manifold and $E$ be a vector bundle of rank $r$ over $M$. Let $i \colon M \hookrightarrow E$ denote the inclusion of $M$ as the zero section of $E$ and $\pi \co E \rightarrow M$ be the projection. Let $E_0 = E - i(M)$. According to the Thom isomorphism, $\pi^*$ makes $\rH^{i+r} (E, E_0; \Z/2)$ into a free $\rH^*(M; \Z/2)$-module of rank 1, with generator the cycle class of the zero-section, which we denote $[M] \in \rH^{r}(E, E_0; \Z/2)$. Then $w_i \in \rH^i(M;\Z/2)$ is the unique class such that 
\[
\Sq^i ([M]) = \pi^*(w_i) \smile [M] \in \rH^{i+r} (E, E_0; \Z/2).
\]
This characterization of $w_i$ can be taken as a definition; see \cite[\S 8]{MS74} for the theory developed in this way. 

With this perspective on the Stiefel--Whitney classes, the proof of Wu's formula is essentially pure linear algebra. 

\begin{proof}[Proof of Theorem \ref{thm:wu}] In this proof, all cohomology is taken with $\Z/2$ coefficients, which we omit from the notation. For a closed subspace $A \inj X$, we abbreviate $\rH^*_A(X; -)$ for the relative cohomology $\rH^*(X, X \setminus A; -)$. 

We note that the normal bundle to the diagonal embedding $M \inj M \times M$ is isomorphic to the tangent bundle of $M$. By excision, we therefore have 
\begin{equation}\label{eq:excision}
\rH^*_M(M \times M) \cong \rH^*_M(TM),
\end{equation}
and both sides are freely generated by the cycle class $[M]$ that is Poincar\'e dual to the homological fundamental class of $M$, embedded as the diagonal in $M \times M$ and as the zero section in $TM$. This gives a commutative diagram, 
\[
\begin{tikzcd}
\rH^*_M(TM) \ar[r, "\sim", "\eqref{eq:excision}"'] \ar[d, "\pi_*"', bend right] & \rH^*_M(M \times M) \ar[d, "\pr_{1*}"', bend right] \ar[r]   & \rH^*(M \times M) \ar[d, "\pr_{1*}"', bend right]  \\
\rH^*(M) \ar[u, "\pi^*"', bend right]  \ar[r, equals] & \rH^*(M) \ar[u, "\pr_1^*"', bend right]  \ar[r, equals] & \rH^*(M) \ar[u, "\pr_1^*"', bend right] 
\end{tikzcd}
\]
The isomorphism \eqref{eq:excision} is equivariant for the Steenrod operations. Since by definition $\Sq^i([M]) = \pi^* (w_i) \smile [M] \in \rH^*_M(TM)$, we also have $\Sq^i([M]) = \pr_1^*(w_i) \smile [M] \in \rH^*_M(M \times M)$. Letting $\pr_{1*} \co \rH^*(M \times M) \rightarrow \rH^*(M)$ be the pushforward map adjoint to $\pr_1^*$ under Poincar\'e duality, this implies that
\begin{equation}\label{eq:Wu-LHS}
w = \pr_{1*} (\Sq ([M])).
\end{equation}
We will proceed to calculate the right side in a different way. 

The K\"unneth formula identifies 
\begin{equation}\label{eq:Kunneth}
\rH^*(M \times M) \cong \rH^*(M) \otimes \rH^*(M).
\end{equation}
Concretely, if we pick a basis $\{e_i\}$ for $\rH^*(M)$ and let $\{f_i\}$ be the dual basis under Poincar\'e duality, then the identification \eqref{eq:Kunneth} carries the cohomology class $[\Delta] \in \rH^*(M \times M)$ to $\sum_i e_i \otimes f_i$. (Recall that we are working with $\Z/2\Z$ coefficients; if we were working with integral coefficients, then we would have to mind certain signs here.) 

The \emph{Cartan formula} tells us how to compute $\Sq$ of a cup product, and says that 
\[
\Sq(\pr_1^*(e) \smile \pr_2^*(f)) = \pr_1^* \Sq(e) \smile \pr_2^* \Sq(f). 
\]
We use this to calculate $\pr_{1*} (\Sq ([\Delta]))$:
\begin{align*}
\pr_{1*} (\Sq ([\Delta])) &= \pr_{1*} \left( \sum_i \pr_1^*(\Sq (e_i)) \smile \pr_2^*(\Sq (f_i)) \right) \\
&= \sum_i \Sq (e_i)  \int_M \Sq (f_i ) \\
&= \sum_i \Sq (e_i)  \int_M (v f_i).
\end{align*}
Since $\{e_i\}$ is a $\Z/2$-basis of $\rH^*(M)$, we may write $v = \sum_i \epsilon_i e_i$ for unique $\epsilon_i \in \Z/2$. Since $\{f_i\}$ is the dual basis to $\{e_i\}$, we have $\int_M (vf_i) = \epsilon_i$, giving 
\[
\pr_{1*} (\Sq ([\Delta]))  = \sum_i \Sq (e_i) \otimes \int_M (v f_i) = \sum_i \epsilon_i \Sq(e_i) = \Sq (v).
\]
Equating this with \eqref{eq:Wu-LHS} completes the proof.
\end{proof}

\section{Tate's Symplecticity Conjecture for $\ell \neq p$}\label{sec:ell-adic}

Building on the discussion of the preceding sections, we will sketch the proof of Theorem \ref{thm:symplecticityneqp}. At this point the ratio of material covered to space will increase dramatically compared to previous sections, so our discussion will abruptly become more informal and impressionistic. The metaphor used in the talk, on which this article is based, was that if the previous sections were at the resolution of a walking tour, then this section will be at the resolution of the view from a train window. 

\subsection{Fast-forwarding to new features}
We will quickly fast-forward through the aspects of the proof of Theorem \ref{thm:symplecticityneqp} which are shared by the topological case, and have already been discussed in \S \ref{sec:linking-form}. 

Let $X$ be a smooth, proper, geometrically connected variety of dimension $2d$ over a finite field $k$ of characteristic $p$. Assume that $p \neq 2$. Then we have a $4d+1$ dimensional Poincar\'e duality for $\rH^*(X)$ as in \S \ref{sssec:arithmetic-cohomology}. As in \S \ref{ssec:linking-form-alternation}, the $\EE_\infty$ Steenrod operation
\[
\Sq^{2d} \co \rH^{2d+1}(X; \Z/2(d)) \rightarrow \rH^{4d+1}(X; \Z/2(2d))
\]
is represented by a unique class $v_{2d} \in \rH^{2d}(X; \Z/2(d))$. By arguments which are formally similar to the proof of Theorem \ref{thm:alternating-linking-form}, one shows (cf. \cite[\S 7]{Feng20}):

\begin{thm}\label{thm:variety-spin-c} Let $X$ be a smooth, proper, geometrically connected variety of dimension $2d$ over a finite field $k$ of characteristic $p$. Then the Artin--Tate pairing on $\rH^{2d+1}(X; \Z_2(d))_{\tors}$ is alternating if and only if $v_{2d}$ lifts to $\rH^{2d}(X; \Z_2(d))$.
\end{thm}

Recall from Example \ref{ex:wu-manifold} that for topological 5-manifolds (corresponding to $d=1$), it was not necessarily the case that $v_{2d}$ lifts to $\rH^{2}(X; \Z_2(1))$. Also, despite longstanding belief to the contrary, Poonen--Stoll \cite{PS99} showed that the Cassels--Tate pairing need not be alternating for a Jacobian variety. Thus, the fact that the Artin--Tate pairing \emph{is} always alternating reflects something special about this setting, which is not witnessed in other closely analogous mathematical contexts. 

\begin{thm}\label{thm:wu-lifts}
In the situation of Theorem \ref{thm:variety-spin-c}, $v_{2d}$ always lifts to $\rH^{2d}(X; \Z_2(d))$.
\end{thm}

Theorem \ref{thm:wu-lifts} is implicitly in \cite[\S 7]{Feng20}. We will explain the ideas of the proof over the rest of the section. First we give a heuristic explanation, which is not a proof but rather a \emph{worldview} via metaphor to topology, in the spirit of \S \ref{ssec:QR}. We focus on the case of a surface $X/k$ ($d=1$). While the topological analogue of quadratic reciprocity---namely the symmetry of the linking number---was geometrically obvious, the topological analogue of symplecticity is not obvious in this case. It is given by Remark \ref{rem:spin}: a 5-manifold has symplectic linking form if and only if it is spin$^c$. In these terms, we can think of Theorem \ref{thm:symplecticityneqp} as asserting that a surface $X/k$ is analogous to a special class within all 5-manifolds: it looks like it is spin$^c$. Why might this be so? 

Referring back to the discussion of \S \ref{sssec:arithmetic-cohomology}, we see that a careful accounting yields the following more precise analogy. Intuitively, we can think of $X \rightarrow \Spec k$ as being analogous to a topological fibration $M \rightarrow S^1$, with the fiber $X_{\ol k}$ behaving like a $4$-manifold $F^{4}$. 
\begin{equation}\label{diag:fibration-analogy}
\begin{tikzcd}
X_{\overline{k}} \ar[rrr, dashed, bend left] \ar[r]  & X \ar[d] \ar[rrr, dashed, bend left] &  & F^{4}  \ar[r] & M^5 \ar[d]  \ar[lll, dashed, bend right] \\ 
 & \Spec k \ar[rrr, dashed, bend left] &  & & S^1 \ar[lll, dashed, bend right] 
\end{tikzcd}
\end{equation}
This perspective makes it clear that $X$ should behave like a 5-manifold fibered over $S^1$. Moreover---using the principle that a $d$-dimensional variety over a separably closed field behaves (at least when examining algebraic variants which are $\ell$-adic for $\ell \neq p$) like a $d$-dimensional complex variety---the fibers should be considered analogous to complex surfaces. Since the tangent bundle of $S^1$ is trivial, this suggests that the characteristic classes of $TM$ should behave like the characteristic classes of a complex manifold. In particular, since complex manifolds have a canonical spin$^c$ structure, given by the first Chern class $c_1$, this suggests that such an $M$ should also be spin$^c$.

\subsection{Arithmetic Wu formula for $p \neq 2$}
The first step towards substantiating the heuristic just explained, for proving Theorem \ref{thm:variety-spin-c}, is to calculate $v_{2d}$ in terms of Stiefel--Whitney classes. In \S \ref{sec:linking-form} we pursued this using ``Wu's formula''. More generally, we can define $v_i$ for any smooth, proper, geometrically connected $X$, as the cohomology class (whose uniqueness and existence are guaranteed by Poincar\'e duality) representing the Steenrod operation $\Sq^i \co \rH^{4d+1-i}(X; \Z/2) \rightarrow \rH^{4d+1}(X; \Z/2)$. 

We also need to define Stiefel--Whitney classes in the algebraic case. The idea for this is to take Thom's characterization of Stiefel--Whitney classes, explained in \S \ref{ssec:wu-formula}, as the definition. Let $\pi \co E \rightarrow X$ be a vector bundle of rank $r$. Then the relative cohomology $\rH^*_X(E; \Z/2)$ is free of rank 1 over $\rH^*(X; \Z/2)$ by purity, with generator the cycle class $[X] \in \rH^{2r}_X(E; \Z/2)$ of the zero section. We can again define $w_i(E) \in \rH^i(X; \Z/2)$ to be the unique class such that $\Sq^i([X]) = \pi^*(w_i) \smile [X]$. 

\begin{thm}[Arithmetic Wu formula, {\cite[Theorem 6.5]{Feng20}}]\label{thm: etale Wu}
Let $X$ be a smooth, proper, geometrically connected variety over $k$. Then we have $\Sq(v) = w$. 
\end{thm} 

The fact that Theorem \ref{thm: etale Wu} looks like a translation of Theorem \ref{thm:wu} belies difficulties in its proof, which is \emph{not} a simple translation of the proof of Theorem \ref{thm:wu}. This is due to the failure of arithmetic cohomology to behave like singular cohomology in some important aspects. The analogous statement for \emph{geometric cohomology} can be proved by translating the topological argument line-by-line. In fact, various generalizations of this analogous ``geometric Wu formula'' have been proved in \cite{Urabe96}, \cite{SS24}, and \cite{Ben24}. We proceed to explain the key complication of the arithmetic situation.

Let us try to follow the proof of the topological Wu formula in Theorem \ref{thm:wu}. Take all cohomology with coefficients in $\Z/2$, and omit the coefficients from the notation. We already see a problem at the step of invoking the K\"unneth formula: there is no K\"unneth formula identifying $\rH^*(X \times_{k} X)$ with $\rH^*(X) \otimes_{\F_2} \rH^*(X)$. (There would have been a K\"unneth formula if $k$ were separably closed.) Indeed, if $X$ has dimension $2d$ then $X \times_{k} X$ has dimension $4d$, so $\rH^*(X \times_{k} X)$ has a Poincar\'e duality structure of dimension $8d+1$, while $\rH^*(X) \otimes_{\F_2} \rH^*(X)$ has a Poincar\'e duality structure of dimension $8d+2$. This latter object does not resemble the cohomology of any natural variety over $k$; it is just an algebraic object that one can formally write down.

There is still a map $\varphi_* \co \rH^*(X \times_{k} X) \rightarrow \rH^*(X) \otimes_{\F_2} \rH^*(X)$, which can be described as follows. Dualizing the cup product
\[
\rH^*(X) \otimes_{\F_2} \rH^*(X) \rightarrow \rH^*(X \times_{k} X)
\]
and then applying Poincar\'{e} duality, we obtain a commutative diagram 
\begin{equation}\label{diag: intro varphi}
\begin{tikzcd}
\rH^*(X)^\vee \otimes_{\F_2} \rH^*(X)^\vee \ar[d, "\wr"]  & \ar[l] \rH^*(X \times_{k} X)^\vee \ar[d, "\wr"]  \\
\rH^*(X) \otimes_{\F_2} \rH^*(X) & \ar[l, "\varphi_*"] \rH^*(X \times_{k} X)
\end{tikzcd}
\end{equation}
Here $(-)^\vee$ denotes linear duality over $\F_2$. 

Another interpretation of $\varphi_*$ is as follows. There is a map $\rH^*(X \times_{k} X)  \rightarrow \End(\rH^*(X))$ sending a cohomology class $\kappa \in \rH^*(X \times_{k} X) $ to the ``kernel'' endomorphism 
\[
u \mapsto \pr_{2*}(\pr_1^*(u) \smile \kappa).
\]
It is clear from the definition that this map still carries the cycle class $[\Delta] \in \rH^*(X \times_{k} X)$ of the diagonal copy of $X$ to the identity endomorphism. Using Poincar\'e duality to identify $\End(\rH^*(X)) \cong \rH^*(X) \otimes \rH^*(X)$, we recover $\varphi_*$ as defined as above. In particular, it is clear that $\varphi_*([\Delta]) = \sum_i e_i \otimes f_i$. To run the rest of the proof, however, we need furthermore that 
\[
\varphi_* (\Sq ([\Delta]))  = \sum_i \Sq(e_i) \otimes \Sq(f_i) = \Sq (\varphi_* ([\Delta])).
\]
This turns out to be quite difficult, and we formulate it as a separate result. 

\begin{prop}[{\cite[Proposition 6.14]{Feng20}}]\label{prop:varphi-equivariant}
The map $\varphi_*$ is equivariant with respect to the Steenrod operations.
\end{prop}

We will give a toy metaphor for Proposition \ref{prop:varphi-equivariant}. Heuristically, $\varphi_*$ behaves like a pushforward map\footnote{Also sometimes called ``Umkehr map'' or ``wrong-way map''.} on cohomology, hence the notation (even though there is no actual map $\varphi$ which induces it). As discussed above, $\Spec k$ is topologically analogous to $S^1$, so that $X$ is topologically analogous to a manifold $M$ fibered over $S^1$. Then $X \times_{k} X$ is analogous to $M \times_{S^1} M$ and $\rH^*(X) \otimes_{\F_2} \rH^*(X)$ is analogous to the cohomology of $M \times M$. In these terms, $\varphi_*$ would be analogous to the pushforward map on cohomology associated to $M \times_{S^1} M \rightarrow M \times M$. Typically, pushforward maps are \emph{not} equivariant with respect to Steenrod operations; the failure of equivariance should be measured by the Stiefel--Whitney classes of the relative normal bundle. But in this particular situation, because $S^1$ has trivial tangent bundle, one would expect equivariance. 

Of course, this discussion is all heuristic: in our actual situation, there is not even a geometric object whose cohomology realizes $\rH^*(X) \otimes_{\F_2} \rH^*(X)$. This exposes a defect in the analogy between varieties over $k$ and topological spaces fibered over $S^1$: in the latter situation one can forget the fibration and consider the bare topological space, while there is no corresponding move for varieties over $k$. Hence any operation performed in the category of varieties over $k$ is really a ``relative'' operation: the product of varieties over $k$ corresponds to the \emph{fibered} product of manifolds over $S^1$, the tangent bundle of a variety over $\F_q$ corresponds to the \emph{relative} tangent bundle over $S^1$, etc. Because of this, there are some steps in the proof of Wu's Theorem that have no analogue in the category of varieties over $k$. 

In \cite{Feng20}, the solution is to use \'etale homotopy theory to suitably approximate the situation by the topological one discussed in the above metaphor. Passing to \'{e}tale homotopy type allows one to disassociate a variety from this fibration, and thus acquire some of the additional flexibility enjoyed by topological spaces. This trick is not available in the $p$-adic case studied in \cite{CF}, so a new and more robust perspective was developed there that should also be applicable to the $\ell$-adic setting, although this has not yet been written down.

\subsection{Relation to Chern classes}\label{ssec:chern-classes} Suppose we grant the arithmetic Wu formula, Theorem \ref{thm: etale Wu}. This allows us to solve for $v_i(X)$ in terms of Stiefel--Whitney classes $w_{\leq i}(TX)$. Then to prove Theorem \ref{thm:wu-lifts}, it suffices to lift the Stiefel--Whitney classes to integral coefficients. 

\begin{thm}\label{thm:sw-lifts}
In the situation of Theorem \ref{thm:wu-lifts}, each $w_{2i}(TX) \in \rH^{2i}(X; \Z/2)$ lifts to $\rH^{2i}(X; \Z_2(i))$.  
\end{thm}

\begin{remark}The inspiration for this lifting comes from the following topological metaphor: for a complex manifold $M$, the Stiefel-Whitney classes are the reductions mod 2 of the Chern classes, and in particular admit integral lifts. Note however that the correct topological analogue of our $X$ is a manifold of dimension $4d+1$, which obviously cannot admit a complex structure, so this principle seems irrelevant at first. 

The resolution is the more refined analogy \eqref{diag:fibration-analogy}, where we pointed out that $X$ is more specifically analogous to a complex manifold fibration $M$ fibered over $S^1$. Since $TS^1$ is trivial, the characteristic classes of $TM$ ``come from'' the complex fiber, and hence behave like those of a complex manifold. 
\end{remark}

Heuristics aside, what actually goes into the proof of Theorem \ref{thm:sw-lifts}? The basic difficulty is that our definition of Stiefel--Whitney classes is difficult to compute explicitly, because it is formulated in terms of Steenrod operations, which are generally difficult to compute explicitly. Our goal is to relate the Stiefel--Whitney classes to reductions of Chern classes, which are relatively concrete. Therefore, our task is effectively to make a computation of Steenrod operations.

What makes this possible is that the formalism of characteristic classes allows one to reduce such identities to the elemental case of line bundles (a manifestation of the ``splitting principle''). Since line bundles have rank one, the cohomology class of the zero section, call it $[X]$, has cohomological degree $2$. Because of its low cohomological degree, we \emph{are} able to compute the Steenrod operations on $[X]$ explicitly. We automatically have $\Sq^i([X])=0$ whenever $i>2$ for degree reasons, $\Sq^0([X]) = [X]$, and $\Sq^2([X]) = [X] \smile [X]$. We also have a concrete description of $\Sq^1$, although this turns out to be slightly subtle because it depends on whether one uses the $\EE_\infty$ or the ``motivic'' incarnation of the Steenrod operations. The paper \cite{Feng20} uses the $\EE_\infty$ operations, for which $\Sq^1$ is the Bockstein operator for the short exact sequence
\[
0 \rightarrow \Z/2 \rightarrow \Z/4 \rightarrow \Z/2 \rightarrow 0.
\]
On the other hand, the motivic incarnation of $\Sq^1$ depends on the weight, and on $\rH^a(X; \Z/2(b))$ it would be
\[
0 \rightarrow \Z/2 (b) \rightarrow \Z/4(b) \rightarrow \Z/2(b) \rightarrow 0.
\]
In either case, $\Sq^1([X])$ is easily calculated by using that $[X]$ lifts to $\rH^2(X; \Z_2(1))$. 

In particular, a canonical lift of $w_{2d}(TX)$ to $\rH^{2d}(X; \Z_2(d))$ can be written down, and it depends on the normalization of $\Sq^1$ as discussed above. Under the motivic normalization, one can show that for any vector bundle $w_{2i+1}(E) = 0$ for all $i  \geq 0$, and $w_{2i} (E) = c_i(E) \pmod{2}$. Under the $\EE_\infty$ normalization taken in \cite{Feng20}, there is a more complicated formula \cite[Theorem 5.10]{Feng20} involving also the Kummer class of $-1$ in $\rH^1(\Spec \F_q; \Z/2(1))$, which can also be viewed as the Bockstein of $-1 \in \rH^0(\Spec \F_q; \mu_2)$. 

\section{Tate's Symplecticity Conjecture for $\ell = p$}\label{sec:ellequalsp}

We finally arrive at the subject of \cite{CF}, which proves Theorem \ref{thm:symplecticityatp}, establishing Conjecture \ref{conj:symplecticity} in the most difficult case, where $p=2$. This difficulty stems from general difficulties of the cohomology in the $\ell = p$ situation. Although the discussion has already become purely impressionistic starting with the previous section, the level of impressionism will increase yet more in this discussion. The metaphor used in the talk, on which this article is based, was that if the previous section was at the resolution of the view from a train window, then this section will be at the resolution of the view from an airplane window.

\subsection{Summary of new obstacles}
Roughly speaking, the hope of \cite{CF} is to execute the strategy of \cite{Feng20} in the $\ell=p$ ($=2$) setting. Let us summarize the new issues which arise when trying to do this. 

The first problem has to do with the Steenrod operations themselves. Although one still has the $\EE_\infty$ Steenrod operations on $\rH^*(X; \Z/2\Z(*))$, which connect to the Milne--Artin--Tate pairing via Theorem \ref{thm: MAT form}, these are no longer the ``right'' operations to use for computing with characteristic classes. This is immediately seen from consideration of weights. When $p \neq 2$, we could identify $\Z/2 = \Z/2(1) = \Z/2(2) = \ldots$ so the weights of cohomology were effectively negligible when working modulo $2$. This is no longer the case for syntomic cohomology, so we have to carefully track weights. Let us start writing $\Sqe^i$ for the $\EE_\infty$ Steenrod operations for emphasis; then each $\Sqe^{2i}$ doubles the weight (\S \ref{ssec:E-infty-steenrod}). This implies for example that $\Sqe^0$ cannot be the identity map, since that would preserve the weight. 

If we tried to use $\Sqe^i$ to define a theory of syntomic Stiefel--Whitney classes as before, then we would find that the resulting $w_i^{\EE}(E) \in \rH^i(X; \Z/2(\rank E))$ lies in a cohomology group whose twist parameter is a function of $\rank(E)$ rather than $i$. This does not match our expectations for a theory of characteristic classes, and is technically incompatible with (say) Chern classes and related ``motivic'' constructions. 

Also, any corresponding ``arithmetic Wu formula'' of the form $w^{\EE} = \Sqe(v)$ could not be inverted to solve for $v$ in terms of $w^{\EE}$, since the $\Sqe$ operation is nilpotent on any finite type $X$, as it doubles the weights.

It turns out that on syntomic cohomology there are two \emph{distinct} incarnations of Steenrod operations. One incarnation comes from an $\EE_\infty$-structure; in this section we denote such operations by $\Sqe^i$. The other incarnation we call the ``syntomic Steenrod operations'' $\Sqs^i$. These operations are of a more motivic nature, and were anticipated by Voevodsky; we construct them in \cite{CF}. We emphasize that we develop the theory of the syntomic Steenrod algebra for all primes $p$, although only the case $p=2$ is needed for Theorem \ref{thm:symplecticityatp}. We discuss the syntomic Steenrod algebra more in \S \ref{ssec:syntomic-steenrod-algebra} below.

After the construction of the syntomic Steenrod algebra, the next major hurdle is the proof of an arithmetic Wu formula. As in the $\ell \neq p$ case, the technical crux is to show that the analogously defined 
\[
\varphi_* \co \rH^*(X \times_{k} X) \rightarrow  \rH^*(X) \otimes_{\F_p} \rH^*(X) 
\]
is equivariant for the syntomic Steenrod operations. The proof requires us to develop a theory of ``spectral prismatic $F$-gauges'', in accordance with Lurie's vision of a ``prismatic stable homotopy theory''. For this task, we received crucial guidance from Lurie.

\subsection{The syntomic Steenrod algebra}\label{ssec:syntomic-steenrod-algebra}  
Geisser--Levine \cite{GL00} proved that for smooth varieties over finite fields, syntomic cohomology coincides with $p$-adic \'etale-motivic cohomology. This contextualizes the issue of constructing syntomic Steenrod operations in terms of a classic problem raised by Voevodsky in his manuscript \cite{Voe02} on ``Open problems in motivic homotopy theory'', which is to \emph{develop a theory of motivic Steenrod operations at the defining characteristic.} 

We recall some historical context for this problem. Away from defining characteristic, the mod $p$ motivic Steenrod algebra for varieties in characteristic $0$ was studied by Voevodsky \cite{Voe03, Voe10}, and for varieties in positive characteristic $\ell \neq p$ by Hoyois--Kelly--Ostvaer \cite{HKO}. Thanks to their work, the mod $p$ motivic Steenrod algebra is now well-understood away from characteristic $p$, but in characteristic $p$ it is still highly mysterious. Voevodsky conjectured a description of it in \cite{Voe02}, which implies that it should have a Milnor basis of power operations over the motivic cohomology of the base field. This conjecture remains wide open, but partial evidence was given by Frankland--Spitzweck \cite{FS18} who showed that Voevodsky's ``expected answer'' for the \emph{dual} motivic Steenrod algebra does at least appear as a (module-theoretic) summand of the true answer. This allows one to define power operations on motivic cohomology, as was suggested already in \cite{FS18} and carried out by Primozic in \cite{Pri20}, but from this definition it seems infeasible to control key technical properties of the operations, such as the Cartan formula. 

In \cite{CF}, the author and Shachar Carmeli took a completely different route, suggested by Lurie, for constructing syntomic Steenrod operations. Concurrently and independently, Annala--Elmanto \cite{AE} used similar ideas to give a new construction of motivic Steenrod operations in defining characteristic. We expect that their operations recover ours upon \'etale sheafification. However, we do not know how to prove the arithmetic Wu formula other than by our method in \cite{CF}, so it is crucial that we use that approach.

\subsection{Perfectoid nearby cycles}
In \cite{Voe03, Voe10} Voevodsky defined and analyzed the motivic Steenrod algebra on mod $p$ motivic cohomology of varieties over characteristic zero fields. The idea for constructing the syntomic Steenrod algebra is to transport Voevodsky's results to characteristic $p$ using a ``perfectoid nearby cycles'' functor \cite[\S 4]{CF} which was suggested to us by Jacob Lurie (and credited by Lurie to Akhil Mathew, who in turn credits Niziol for the inspiration).

We set up some language needed to describe this ``perfectoid nearby cycles'' functor. For a scheme $S$, let $\SH_S$ be Morel--Voevodsky's \emph{$p$-complete motivic stable homotopy category} of $\A^1$-invariant cohomology theories over $S$. Thus $\SH_S$ contains an object $\MHFp$ representing mod $p$ motivic cohomology. Annala--Hoyois--Iwasa \cite{AHI1} constructed the category of $p$-complete \emph{motivic spectra} $\MS_S$, an enlargement of $\SH_S$ which includes non-$\A^1$-invariant cohomology theories. In particular, mod $p$ syntomic cohomology (which is not $\A^1$-invariant) promotes to an object $\MSFp \in \MS_S$. The objects of $\MS_S$ give rise to Nisnevich sheaves of ($p$-complete) \emph{spectra} on smooth schemes over $S$. If $S = \Spec R$, we will also denote $\MS_R:= \MS_S$ and $\SH_R := \SH_S$.

We start by considering the integral perfectoid ring $\Zpcyc = \Z_p[\mu_{p^\infty}]^{\wedge}_p$, obtained by adjoining all $p$-power roots of unity to $\Z_p$ and then $p$-adically completing. Then associated to 
\[
j \co \Spec \Q_p^{\cyc} \inj \Spec \Z_p^{\cyc} \hookleftarrow \Spec \F_p  \co i
\]
we obtain functors
\[
\begin{tikzcd}
\MS_{\Q_p^{\cyc}} \ar[r, "j_*"] & \MS_{\Z_p^{\cyc}} \ar[r, "i^*"] & \MS_{\F_p} 
\end{tikzcd}
\]
The ``perfectoid nearby cycles'' functor $\psi \co \SH_{\Q_p^{\cyc}} \rightarrow \MS_{\F_p}$ is defined as the $i^* L_{\et} j_*$, where $L_{\et}$ is the \'etale sheafification functor. The \emph{key calculation}, which has moral roots in \cite{AMM22} and \cite{Niz98}, is that $\psi$ carries the motivic cohomology spectrum $\MHFp \in \SH_{\Q_p^{\cyc}}$ to the syntomic cohomology spectrum $\MSFp \in \MS_{\F_p}$. This fact ultimately allows us to transmute information about the motivic Steenrod algebra in characteristic $0$, which was explicated by Voevodsky in \cite{Voe03, Voe10}, into information about the syntomic Steenrod algebra in characteristic $p$. Indeed, the motivic Steenrod algebra over $K$ can be defined as 
\[
\cA_{\mot}^{*,*} := \Ext^{*,*}_{\Sph}(\MHFp, \MHFp)
\]
where $\Sph \in \SH_K$ is the symmetric monoidal unit, which is the \emph{$p$-complete motivic sphere spectrum}. Taking inspiration from the key calculation mentioned above, we may regard $\psi(\Sph)$ as ``the syntomic sphere spectrum'' over $k$, and then try to define a syntomic Steenrod algebra as
\[
\Ext^{*,*}_{\psi(\Sph)}(\psi(\MHFp), \psi(\MHFp)).
\]
This is a technically oversimplified but morally faithful approximation to our actual definition of the syntomic Steenrod algebra $\cAsyn^{*,*}$.

\subsection{Spectral prismatization}\label{intro: spectral prismatization} Suppose that the syntomic Steenrod operations have been satisfactorily constructed. The next step is to define syntomic Stiefel--Whitney classes, following the template of \S \ref{ssec:wu-formula}. Then one needs to prove an arithmetic Wu formula, which looks cosmetically the same as Theorem \ref{thm:wu-lifts}. We stress again that the difficulty comes from the ``arithmetic direction''; for a geometric Wu formula one can see \cite[\S 4.4]{AE}. 

For the arithmetic Wu formula, the key technical difficulty is to show compatibility of a map $\varphi_*$, defined analogously as in \eqref{diag: intro varphi}, with respect to the syntomic Steenrod operations. For this, we cannot mimic the trick of \cite{Feng20}, which used \'etale homotopy theory to reduce to a topological situation.

Instead, \cite{CF} develops a new approach to this problem, which involves building an apparatus that is likely of broader importance than the application itself: a generalization of \emph{prismatization} in the sense of Drinfeld and Bhatt--Lurie, for the syntomic sphere spectrum, and an attendant ``spectral Serre duality''. In particular, we construct an approximation to Lurie's conjectural \emph{prismatic stable homotopy category} over $k$, which is an extension of the category of prismatic $F$-gauges over $k$ in the spirit of stable homotopy theory. To be clear, the idea for how to do this was explained to us by Lurie. We emphasize that we develop this story \emph{for general $p$}, although only the case $p=2$ is needed for the proof of Theorem \ref{thm:symplecticityatp}.

Let us indicate how this machinery enters the story. Looking at the formulation of Proposition \ref{prop:varphi-equivariant}, we see that it is fundamentally about the compatibility between two phenomena: Poincar\'e duality, and Steenrod operations. These two phenomena are of very different origin, at least historically. Milne's proof of Poincar\'e duality used the relatively concrete description of his logarithmic de Rham--Witt cohomology groups in terms of differential forms. On the other hand, the syntomic Steenrod operations arose very indirectly, via perfectoid nearby cycles. Even the comparison \cite[Corollary 8.21]{BMS2} between the underlying cohomology theories---Milne's logarithmic de Rham--Witt cohomology theory and Bhatt--Morrow--Scholze's syntomic cohomology theory---is quite complicated. Therefore, in order to analyze the compatibility, it is natural to seek a new perspective on Poincar\'e duality, which is native to the formalism of syntomic cohomology in \cite{BL22a} that we used to construct the syntomic Steenrod algebra. 

Fortunately, such a perspective may be found in \cite{Bha22}, which reinterprets Poincar\'e duality for syntomic cohomology in terms of Serre duality for quasicoherent sheaves on a certain stack $\FSyn{\F_p}$, introduced by Drinfeld \cite{Drin24} and Bhatt--Lurie \cite{Bha22}. This recasts the problem in terms of the compatibility of Serre duality on $\FSyn{\F_p}$ with the syntomic Steenrod operations. 

The category of mod $p$ prismatic $F$-gauges over $k$ is
\[
 \FGauge{\FF_p} := \QCoh(\FSyn{\F_p}).
 \]
The syntomic Steenrod algebra can be interpreted as the ring of derived endomorphisms of the structure sheaf $\cO_{\FSyn{\F_p}}$ within a ``thickened'' category of \emph{spectral prismatic $F$-gauges} $\pFGauge{\mathbb{S}}$, analogously to how the classical Steenrod algebra in topology can be interpreted as the ring of endomorphisms of $\F_p$ in the category of spectra. (Here the notation $\mbb{S}$ has to do with the motivic sphere spectrum.) We construct the category $\pFGauge{\mathbb{S}}$ following guidance of Lurie; it is closely related to his conjectural ``prismatic stable homotopy category'' (over $k$).\footnote{The superscript ``pre'' reflects that it is only an approximation to the latter, which however is an equivalence on all the objects that we consider in this paper (see \cite[Remark 9.2.4]{CF}) for a precise explanation of this statement.}

This allows us to lift the syntomic Steenrod to a quasicoherent sheaf of algebras in $\QCoh(\FSyn{\F_p})$ that we call the ``prismatized syntomic Steenrod algebra''. Then we can ``localize'' the equivariance question onto $\FSyn{\F_p}$. To explain this: there is an adjunction 
\[
\iota^*\colon \pFGauge{\mbb{S}} \adj \FGauge{\FF_p} : \iota_*
\]
which should morally be thought of as coming from an embedding $\iota$ of $\FSyn{\F_p}$ into a ``spectral prismatization stack'' $\FSyn{\sph}$. We then define the internal Hom algebra in $\pFGauge{\mathbb{S}}$, 
\[
\sAsyn := \cRHom_{\pFGauge{\mathbb{S}}}(\iota_* \cO_{\FSyn{\F_p}} , \iota_* \cO_{\FSyn{\F_p}}),
\]
as the ``prismatization of the syntomic Steenrod algebra''; it recovers the syntomic Steenrod algebra $\Asyn^{*,*}$ upon taking global sections. By construction, all objects of $\FGauge{\FF_p}$ which arise via $\iota^*$ from $\FSyn{\sph}$ are equipped with a tautological action of $\sAsyn$.

The situation is formally analogous to that of a $G$-torsor $\pi \co P \rightarrow S$. Since $G$ is the group of endomorphisms of $P$ over $S$, an object over $S$ pulls back to a $G$-equivariant object over $P$. Analogously, an object on $\FSyn{\sph}$ pulls back to an $\sAsyn$-equivariant object on $\FGauge{\FF_p}$. Conversely, in the situation of $\pi \co P \rightarrow S$, to prove that a structure over $P$ is $G$-equivariant amounts to descending it to $S$. Analogously, to prove that the syntomic Steenrod action is compatible with Serre duality on $\FSyn{\F_p}$, the crux is to appropriately extend Serre duality to a ``spectral Serre duality'' on $\pFGauge{\mathbb{S}}$. (See \cite{SHA} for a more abstract discussion which formally puts these two situations in a uniform framework.)

The key to spectral Serre duality is to find the right definition of the dualizing sheaf in $\pFGauge{\mathbb{S}}$. For this we were again aided by a suggestion of Lurie to take inspiration from \emph{Brown--Comenetz duality} \cite{BC76}, which is a generalization of Pontrjagin duality to spectra. To be precise, let $\Sp$ be the category of $p$-complete spectra and $\bI \in \Sp$ be the $p$-completion of the Brown--Comenetz spectrum. We ``pull back'' (in a suitable sense) $\bI$ from $\Sp$ to define a dualizing sheaf in $\pFGauge{\mbb{S}}$, which we show fits into a good theory of ``spectral Serre duality'' on $\pFGauge{\mbb{S}}$, compatible with the classical theory of coherent duality in $\FGauge{\FF_p}$. This interplay can be used to show the desired compatibility of the prismatized Steenrod algebra action with Serre duality, as explained in \cite[\S 11]{CF}.

\subsection{The comparison theorem} Theorem \ref{thm: MAT form} connects the Milne--Artin--Tate pairing to $\EE_\infty$ Steenrod operations. However, it is the \emph{syntomic} Steenrod operations that we can calculate effectively, using the arithmetic Wu formula and the connection between Stiefel--Whitney classes and Chern classes, established similarly as in \S \ref{ssec:chern-classes}. Now there is a disconnect between these two stories, and to bridge it we need to relate the two types of Steenrod operations. They obviously differ in general because they have different effects on the weights, e.g., 
\begin{align}\label{eq: intro E weight change}
\Sqe^{2i}  &\co \rH^a(-; \Z/2(b)) \rightarrow \rH^{a+2i}(-; \Z/2(2b)) \\
\Sqs^{2i} & \co \rH^a(-; \Z/2(b)) \rightarrow \rH^{a+2i}(-; \Z/2(b+i)).
\end{align}
Note, however, that the codomains \emph{sometimes} agree. For example, the crucial instance for Theorem \ref{thm:symplecticityatp} is the case $p=2$, $a=3$, and $b=1$, where $\Sqe^2$ and $\Sqs^2$ ``coincidentally'' both take the form $\rH^{3}(-; \Z/2(1)) \rightarrow \rH^5(-;\Z/2(2))$. We prove the following ``Comparison Theorem'' asserting that the two flavors of operations agree \emph{whenever} they have the same domain and codomain. 

\begin{thm}[{\cite[Theorem 8.1.1]{CF}}]\label{thm: intro compare operations}
If $b = i$, so that the two maps in \eqref{eq: intro E weight change} have the same source and target, then they agree. 
\end{thm}

In fact, {\cite[Theorem 8.1.1]{CF}} determines the relationship between $\Sqe^i$ and $\Sqs^i$ in all cases. The main content of Theorem \ref{thm: intro compare operations} is an analogous result of Bachmann--Hopkins \cite{BH25} for motivic cohomology, in characteristic 0, which we bootstrap to characteristic $p$ using our perfectoid nearby cycles functor.

This completes our vague and impressionistic tour of the proof of Theorem \ref{thm:symplecticityatp}. Figure \ref{fig:roadmap} gives a bird's eye view of the story.

\bibliographystyle{amsalpha}

\bibliography{Bibliography}

\providecommand{\bysame}{\leavevmode\hbox to3em{\hrulefill}\thinspace}
\providecommand{\MR}{\relax\ifhmode\unskip\space\fi MR }
\providecommand{\MRhref}[2]{%
  \href{http://www.ams.org/mathscinet-getitem?mr=#1}{#2}
}
\providecommand{\href}[2]{#2}
\begin{thebibliography}{HKOsr17}

\bibitem[AE25]{AE}
Toni Annala and Elden Elmanto, \emph{Motivic power operations at the
  characteristic via inﬁnite ramiﬁcation}, 2025.

\bibitem[AHI25]{AHI1}
Toni Annala, Marc Hoyois, and Ryomei Iwasa, \emph{Algebraic cobordism and a
  {C}onner-{F}loyd isomorphism for algebraic {K}-theory}, J. Amer. Math. Soc.
  \textbf{38} (2025), no.~1, 243--289.

\bibitem[AMM22]{AMM22}
Benjamin Antieau, Akhil Mathew, and Matthew Morrow, \emph{The {K}-theory of
  perfectoid rings}, Doc. Math. \textbf{27} (2022), 1923--1952.

\bibitem[BC76]{BC76}
Edgar~H. Brown, Jr. and Michael Comenetz, \emph{Pontrjagin duality for
  generalized homology and cohomology theories}, Amer. J. Math. \textbf{98}
  (1976), no.~1, 1--27.

\bibitem[Ben25]{Ben24}
Olivier Benoist, \emph{Steenrod operations and algebraic classes}, Tunis. J.
  Math. \textbf{7} (2025), no.~1, 53--89.

\bibitem[BH25]{BH25}
Tom Bachmann and Michael Hopkins, \emph{Stable operations in motivic homotopy
  theory}, 2025, Available at \url{https://tom-bachmann.com/ops.pdf}.

\bibitem[Bha22]{Bha22}
Bhargav Bhatt, \emph{Prismatic ${F}$-gauges}, 2022, Lecture notes, available at
  \url{https://www.math.ias.edu/~bhatt/teaching/mat549f22/lectures.pdf}.

\bibitem[BL22]{BL22a}
Bhargav Bhatt and Jacob Lurie, \emph{Absolute prismatic cohomology}, 2022.

\bibitem[BMS19]{BMS2}
Bhargav Bhatt, Matthew Morrow, and Peter Scholze, \emph{Topological
  {H}ochschild homology and integral {$p$}-adic {H}odge theory}, Publ. Math.
  Inst. Hautes \'{E}tudes Sci. \textbf{129} (2019), 199--310.

\bibitem[CF25]{CF}
Shachar Carmeli and Tony Feng, \emph{{Prismatic Steenrod operations and
  arithmetic duality on Brauer groups}}, 2025.

\bibitem[Dri24]{Drin24}
Vladimir Drinfeld, \emph{Prismatization}, 2024.

\bibitem[Fen20a]{Feng20}
Tony Feng, \emph{\'{E}tale {S}teenrod operations and the {A}rtin-{T}ate
  pairing}, Compos. Math. \textbf{156} (2020), no.~7, 1476--1515.

\bibitem[Fen20b]{SHA}
Tony Feng, \emph{The {S}pectral {H}ecke {A}lgebra}, 2020.

\bibitem[FS18]{FS18}
Martin Frankland and Markus Spitzweck, \emph{Towards the dual motivic
  {S}teenrod algebra in positive characteristic}, 2018.

\bibitem[Gei20]{Gei20}
Thomas~H. Geisser, \emph{Comparing the {B}rauer group to the
  {T}ate-{S}hafarevich group}, J. Inst. Math. Jussieu \textbf{19} (2020),
  no.~3, 965--970.

\bibitem[GL00]{GL00}
Thomas Geisser and Marc Levine, \emph{The {$K$}-theory of fields in
  characteristic {$p$}}, Invent. Math. \textbf{139} (2000), no.~3, 459--493.

\bibitem[Gor79]{Go79}
W.~J. Gordon, \emph{Linking the conjectures of {A}rtin-{T}ate and
  {B}irch-{S}winnerton-{D}yer}, Compositio Math. \textbf{38} (1979), no.~2,
  163--199.

\bibitem[HKOsr17]{HKO}
Marc Hoyois, Shane Kelly, and Paul~Arne \O~stv\ae r, \emph{The motivic
  {S}teenrod algebra in positive characteristic}, J. Eur. Math. Soc. (JEMS)
  \textbf{19} (2017), no.~12, 3813--3849.

\bibitem[Ill79]{Ill79}
Luc Illusie, \emph{Complexe de de {R}ham-{W}itt et cohomologie cristalline},
  Ann. Sci. \'{E}cole Norm. Sup. (4) \textbf{12} (1979), no.~4, 501--661.

\bibitem[KT03]{KT}
Kazuya Kato and Fabien Trihan, \emph{On the conjectures of {B}irch and
  {S}winnerton-{D}yer in characteristic {$p>0$}}, Invent. Math. \textbf{153}
  (2003), no.~3, 537--592.

\bibitem[LLR05]{LLR05}
Qing Liu, Dino Lorenzini, and Michel Raynaud, \emph{On the {B}rauer group of a
  surface}, Invent. Math. \textbf{159} (2005), no.~3, 673--676.

\bibitem[LLR18]{LLR18}
\bysame, \emph{Corrigendum to {N}\'{e}ron models, {L}ie algebras, and reduction
  of curves of genus one and {T}he {B}rauer group of a surface [ {MR}2092767;
  {MR}2125738]}, Invent. Math. \textbf{214} (2018), no.~1, 593--604.

\bibitem[Man67]{Manin67}
Ju.~I. Manin, \emph{Rational surfaces over perfect fields. {II}}, Mat. Sb.
  (N.S.) \textbf{72 (114)} (1967), 161--192.

\bibitem[Man86]{Manin86}
Yu.~I. Manin, \emph{Cubic forms}, second ed., North-Holland Mathematical
  Library, vol.~4, North-Holland Publishing Co., Amsterdam, 1986, Algebra,
  geometry, arithmetic, Translated from the Russian by M. Hazewinkel.

\bibitem[May70]{May70}
J.~Peter May, \emph{A general algebraic approach to {S}teenrod operations}, The
  {S}teenrod {A}lgebra and its {A}pplications ({P}roc. {C}onf. to {C}elebrate
  {N}. {E}. {S}teenrod's {S}ixtieth {B}irthday, {B}attelle {M}emorial {I}nst.,
  {C}olumbus, {O}hio, 1970), Lecture Notes in Mathematics, Vol. 168, Springer,
  Berlin, 1970, pp.~153--231.

\bibitem[Mil75]{Milne75}
J.~S. Milne, \emph{On a conjecture of {A}rtin and {T}ate}, Ann. of Math. (2)
  \textbf{102} (1975), no.~3, 517--533.

\bibitem[Mil86]{Milne86}
\bysame, \emph{Values of zeta functions of varieties over finite fields}, Amer.
  J. Math. \textbf{108} (1986), no.~2, 297--360.

\bibitem[Mil97]{Mil97}
John~W. Milnor, \emph{Topology from the differentiable viewpoint}, Princeton
  Landmarks in Mathematics, Princeton University Press, Princeton, NJ, 1997,
  Based on notes by David W. Weaver, Revised reprint of the 1965 original.

\bibitem[Mil25]{milne2025arithmeticduality}
James~S. Milne, \emph{Arithmetic duality}, 2025.

\bibitem[Mor24]{Mor24}
Masanori Morishita, \emph{Knots and primes---an introduction to arithmetic
  topology}, Universitext, Springer, Singapore, [2024] \copyright 2024, Second
  edition [of 2905431].

\bibitem[MS74]{MS74}
John~W. Milnor and James~D. Stasheff, \emph{Characteristic classes}, Princeton
  University Press, Princeton, N. J.; University of Tokyo Press, Tokyo, 1974,
  Annals of Mathematics Studies, No. 76.

\bibitem[MT68]{MoTan68}
Robert~E. Mosher and Martin~C. Tangora, \emph{Cohomology operations and
  applications in homotopy theory}, Harper \& Row, Publishers, New York-London,
  1968.

\bibitem[Niz98]{Niz98}
Wies\l~awa Nizio\l, \emph{Crystalline conjecture via {$K$}-theory}, Ann. Sci.
  \'{E}cole Norm. Sup. (4) \textbf{31} (1998), no.~5, 659--681.

\bibitem[Pri20]{Pri20}
Eric Primozic, \emph{Motivic {S}teenrod operations in characteristic {$p$}},
  Forum Math. Sigma \textbf{8} (2020), Paper No. e52, 25.

\bibitem[PS99]{PS99}
Bjorn Poonen and Michael Stoll, \emph{The {C}assels-{T}ate pairing on polarized
  abelian varieties}, Ann. of Math. (2) \textbf{150} (1999), no.~3, 1109--1149.

\bibitem[SS24]{SS24}
Federico Scavia and Fumiaki Suzuki, \emph{Two coniveau filtrations and
  algebraic equivalence over finite fields}, 2024.

\bibitem[Ste47]{Ste47}
N.~E. Steenrod, \emph{Products of cocycles and extensions of mappings}, Ann. of
  Math. (2) \textbf{48} (1947), 290--320.

\bibitem[Tat63]{Tate62}
John Tate, \emph{Duality theorems in {G}alois cohomology over number fields},
  Proc. {I}nternat. {C}ongr. {M}athematicians ({S}tockholm, 1962), Inst.
  Mittag-Leffler, Djursholm, 1963, pp.~288--295.

\bibitem[Tat95]{Tate66}
\bysame, \emph{On the conjectures of {B}irch and {S}winnerton-{D}yer and a
  geometric analog}, S\'eminaire {B}ourbaki, {V}ol.\ 9, Soc. Math. France,
  Paris, 1995, pp.~Exp.\ No.\ 306, 415--440.

\bibitem[Tho52]{Thom51}
Ren{\'e} Thom, \emph{Espaces fibr\'es en sph\`eres et carr\'es de {S}teenrod},
  Ann. Sci. Ecole Norm. Sup. (3) \textbf{69} (1952), 109--182.

\bibitem[Ura96]{Urabe96}
Tohsuke Urabe, \emph{The bilinear form of the {B}rauer group of a surface},
  Invent. Math. \textbf{125} (1996), no.~3, 557--585.

\bibitem[Voe02]{Voe02}
Vladimir Voevodsky, \emph{Open problems in the motivic stable homotopy theory.
  {I}}, Motives, polylogarithms and {H}odge theory, {P}art {I} ({I}rvine, {CA},
  1998), Int. Press Lect. Ser., vol.~3, Int. Press, Somerville, MA, 2002,
  pp.~3--34.

\bibitem[Voe03]{Voe03}
\bysame, \emph{Reduced power operations in motivic cohomology}, Publ. Math.
  Inst. Hautes \'{E}tudes Sci. (2003), no.~98, 1--57.

\bibitem[Voe10]{Voe10}
\bysame, \emph{Motivic {E}ilenberg-{M}ac{L}ane spaces}, Publ. Math. Inst.
  Hautes \'{E}tudes Sci. (2010), no.~112, 1--99.

\bibitem[Wal62]{Wall62}
C.~T.~C. Wall, \emph{Killing the middle homotopy groups of odd dimensional
  manifolds}, Trans. Amer. Math. Soc. \textbf{103} (1962), 421--433.

\bibitem[Wu50]{Wu50}
Wen-ts\"{u}n Wu, \emph{Classes caract\'{e}ristiques et {$i$}-carr\'{e}s d'une
  vari\'{e}t\'{e}}, C. R. Acad. Sci. Paris \textbf{230} (1950), 508--511.

\bibitem[Zar89]{Zar}
Yu.~G Zarhin, \emph{The {B}rauer group of a surface over a finite field
  ({R}ussian)}, Arithmetic and Geometry of Varieties (1989), 57--67.

\end{thebibliography}

\end{document}